\documentclass{daj}

\dajAUTHORdetails{%
  title = {The Magnitude of Odd Balls via Hankel Determinants of Reverse Bessel Polynomials}, %% please capitalize all significant words
  author = {Simon Willerton},
  plaintextauthor = {Simon Willerton},
}   %%% END \dajAUTHORdetails

\dajEDITORdetails{%
   year={2020},
   number={5},
   received={30 July 2018},   % received date: example: 7 January 2017
   published={12 May 2020},  % published date
   doi={10.19086/da.12649},       % XXX = number of paper, e.g. da006 for paper#6
}   %%% END \dajEDITORdetails

\pdfoutput=1
\PassOptionsToPackage{breaklinks,hidelinks}{hyperref}

 \usepackage{amsmath}
  \usepackage{amsthm}

  \usepackage{amssymb}
  \usepackage{mathtools}
\usepackage{tikz}
\usepackage{pgfplots}
\usepackage{enumerate}

\allowdisplaybreaks
   \newtheorem{thm}{Theorem}
\newtheorem*{mainthm}{The Main Theorem}
      \newtheorem{prop}[thm]{Proposition}
      \newtheorem{cor}[thm]{Corollary}
      \newtheorem{lemma}[thm]{Lemma}

  \theoremstyle{definition}

\newcommand{\dd}{\mathrm{d}}
\newcommand{\vect}[1]{\vec{#1}}

\newcommand{\super}{\mathsf{sf}}

\newcommand{\define}[1]{\textbf{#1}}
\newcommand{\R}{\mathbb{R}}

\newcommand{\Z}{\mathbb{Z}}
\newcommand{\N}{\mathbb{N}}
\newcommand{\ball}{B_R^n}
\newcommand{\sphere}{S^{n-1}_R}
\newcommand{\dnormal}{\frac{\partial}{\partial\nu}}
\newcommand{\dnormali}[1]{\frac{\partial^#1}{\partial\nu^#1}}
\renewcommand{\vec}{\mathbf}
\newcommand{\psit}{\tau}
\newcommand{\deltap}{\delta}

\newcommand{\one}{\mathbf{1}}

\newcommand{\ring}{\mathcal{R}}

\newcommand{\tempdim}{\kappa}
\DeclareMathOperator{\vol}{Vol}

\usepackage{letltxmacro}
\LetLtxMacro\orgvdots\vdots
\LetLtxMacro\orgddots\ddots

\makeatletter
\DeclareRobustCommand\vdots{%
  \mathpalette\@vdots{}%
}
\newcommand*{\@vdots}[2]{%
  \sbox0{$#1\cdotp\cdotp\cdotp\m@th$}%
  \sbox2{$#1.\m@th$}%
  \vbox{%
    \dimen@=\wd0 %
    \advance\dimen@ -3\ht2 %
    \kern.5\dimen@
    \dimen@=\wd2 %
    \advance\dimen@ -\ht2 %
    \dimen2=\wd0 %
    \advance\dimen2 -\dimen@
    \vbox to \dimen2{%
      \offinterlineskip
      \copy2 \vfill\copy2 \vfill\copy2 %
    }%
  }%
}
\DeclareRobustCommand\ddots{%
  \mathinner{%
    \mathpalette\@ddots{}%
    \mkern\thinmuskip
  }%
}
\newcommand*{\@ddots}[2]{%
  \sbox0{$#1\cdotp\cdotp\cdotp\m@th$}%
  \sbox2{$#1.\m@th$}%
  \vbox{%
    \dimen@=\wd0 %
    \advance\dimen@ -3\ht2 %
    \kern.5\dimen@
    \dimen@=\wd2 %
    \advance\dimen@ -\ht2 %
    \dimen2=\wd0 %
    \advance\dimen2 -\dimen@
    \vbox to \dimen2{%
      \offinterlineskip
      \hbox{$#1\mathpunct{.}\m@th$}%
      \vfill
      \hbox{$#1\mathpunct{\kern\wd2}\mathpunct{.}\m@th$}%
      \vfill
      \hbox{$#1\mathpunct{\kern\wd2}\mathpunct{\kern\wd2}\mathpunct{.}\m@th$}%
    }%
  }%
}
\makeatother

\DeclareMathOperator{\interior}{int}
\newcommand{\e}{\mathrm{e}}

\begin{document}

\begin{frontmatter}[classification=text]
%% EDITOR: this will force the keywords to appear right after the Abstract.
%%   If the abstract is too long and would force the keywords off the
%%   front page, please comment out % [classification=text] above
%%   This way the keywords will be floated on the bottom of the first page
%%   even though the Abstract spills over to the next page.

%%% AUTHOR: Title goes here.  This line is optional.  You must use it
%%   if title has footnote attached or requires nontrivial typesetting,
%%   e.g., inclusion of linebreaks to force nice layout.
%\title{The Magnitude of Odd Balls\\ via Hankel Determinants\\ of Reverse Bessel Polynomials} %% please capitalize all significant words

%%% AUTHOR:
%%% List all authors. If you wish, place grant acknowledgements in \thanks.
%%% In brackets include a short tag for each author.
\author[simon]{Simon Willerton}

%%% AUTHOR: Abstract goes here
\begin{abstract}
Magnitude is an invariant of metric spaces with origins in category theory.  Using potential theoretic methods, Barcel\'o and Carbery gave an algorithm for calculating the magnitude of any odd dimensional ball in Euclidean space, and they proved that it was a rational function of the radius of the ball.  In this paper an explicit formula is given for the magnitude of each odd dimensional ball in terms of a ratio of Hankel determinants of reverse Bessel polynomials.  This is done by finding a distribution on the ball which solves the weight equations.  Using Schr\"oder paths and a continued fraction expansion for the generating function of the reverse Bessel polynomials, combinatorial formulae are given for the numerator and denominator of the magnitude of each odd dimensional ball.  These formulae are then used to prove facts about the magnitude such as its asymptotic behaviour as the radius of the ball grows.
\end{abstract}
\end{frontmatter}

%%% AUTHOR: body of paper starts here
\section{Introduction}
\subsection{Background}
Magnitude was originally introduced by Leinster~\cite{Leinster:Magnitude} as a measure of the `size' of finite metric spaces by formal analogy with his definition of Euler characteristic of finite categories.  It soon became clear~\cite{LeinsterWillerton:AsymptoticMagnitude} that this definition could be extended to many compact infinite metric spaces, such as compact subsets of Euclidean space and this was formalized by Meckes~\cite{Meckes:PositiveDefinite}.  (See Section~\ref{sec:Weightings} for more details on the definition of magnitude.)  For instance, the magnitude of the line segment of length $2R$ --- also known as the $1$-ball of radius $R$ --- is given by $R+1$.

Despite its abstract origins, magnitude has been found to have connections with many areas of mathematics.  For instance, in~\cite{LeinsterCobbold:Diversity} deep connections with biodiversity measurement were given; in~\cite{Willerton:SpheresSurfaces} differential geometric methods were used to relate magnitude to curvature; in~\cite{Meckes:MagnitudeDimensions} analytic methods were used to relate magnitude to Minkowski dimension and potential theory; in~\cite{HepworthWillerton:MagnitudeGraphHomology} homological algebraic methods were used to relate magnitude to graph theory and categorification.

Various computations and calculations led to questions about whether this Euler characteristic-like invariant satisfied an inclusion-exclusion principle; such a thing was certainly satisfied in dimension $1$, but in higher dimension it seemed as though it might work only asymptotically as the metric was scaled up.  Numerical computations were inconclusive and the exact magnitude of any Euclidean subset of dimension greater than $1$ was unknown.

In a recent paper~\cite{BarceloCarbery}, Barcel\'o and Carbery showed, amongst other things, that the magnitude of any odd dimensional ball could be calculated using ideas of potential theory introduced by Meckes~\cite{Meckes:MagnitudeDimensions}.  In particular, they showed how the magnitude potential of such a ball is a solution to a boundary value problem which they were able to reduce to a linear system, and then they applied a recursive method to compute the magnitude of the odd ball, which for fixed dimension is seen to be a rational function of the radius of the ball, and this rational function has integer coefficients.
%For dimensions $1$, $3$, $5$ and $7$ they performed their calculations by hand; it is straightforward to implement their method in a computer algebra system such as Sage, and an observation of Leinster and Meckes~\cite{LeinsterMeckes} allows a simplification of the calculation.

Barcel\'o and Carbery calculated the following by hand using their algorithm, where $\left|B^n_R\right|$ denotes the magnitude of the $n$-dimensional ball of radius~$R$.
\begin{align*}
\left|B^ 1 _R\right|
&=
R+1
\\
\left|B^ 3 _R\right|
&=
%\frac{1}{3!} \, (R^{3} + 6R^{2} + 12  R + 6)
\frac{R^{3} + 6R^{2} + 12  R + 6}{3!}
\\
\left|B^ 5 _R\right|
&=
\frac{R^{6} + 18 \, R^{5} + 135 \, R^{4} + 525 \, R^{3} + 1080 \, R^{2} + 1080 \, R + 360}{5! \, {\left(R + 3\right)}}
\\
\left|B^ 7 _R\right|
&=
%\frac{ \scriptscriptstyle R^{10} + %20(2  R^{9} + 36  R^{8} + 381  R^{7} + 2604 R^{6} + 11907  R^{5} + 36540  R^{4} + 73395 R^{3} + 90720  R^{2} + 60480  R + 1512)
%40  R^{9} + 720  R^{8} + 7620  R^{7} + 52080 R^{6} + 238140  R^{5} + 730800  R^{4} + 1467900  R^{3} + 1814400  R^{2} + 1209600  R + 302400
\frac{  R^{10} +
40  R^{9} + 720  R^{8} +  \dots %7620  R^{7}+ 52080 R^{6} + 238140  R^{5} + 730800  R^{4} + 1467900  R^{3} 
+ 1814400  R^{2} + 1209600  R + 302400
}{7!\, {\left(R^{3} + 12  R^{2} + 48  R + 60\right)}}
%\left|B^ 9 _R\right|
%&= \text{\footnotesize}
%\frac{\tiny R^{15} + 75 \, R^{14} + 2625 \, R^{13} + 56700 \, R^{12} + 842625 \, R^{11} + 9096885 \, R^{10} + 73458000 \, R^{9} + 450198000 \, R^{8} + 2103003000 \, R^{7} + 7453782000 \, R^{6} + 19772764200 \, R^{5} + 38281005000 \, R^{4} + 51854985000 \, R^{3} + 45722880000 \, R^{2} + 22861440000 \, R + 4572288000}{362880 \, {\left(R^{6} + 30 \, R^{5} + 375 \, R^{4} + 2475 \, R^{3} + 9000 \, R^{2} + 16920 \, R + 12600\right)}}
\end{align*}
Barcel\'o and Carbery observed, but did not prove, that the coefficients are non-negative; one might hope that if they are positive integer coefficients then they are actually counting something.  They also claimed that they could bound the degree of the denominator by  $\frac{3n^2-2n+7}{8}$.  We shall prove stronger statements and give combinatorial formulae for the coefficients of the numerator and denominator.  They used Fourier analysis to prove that $\left|B^ n _R\right|\to 1$ as $R\to 0$ and that $\left|B^ n _R\right|= R^n/n! +O(R^{n-1})$ as $R\to\infty$; both of these statements will be consequences of our formulae for the numerator and denominator. 

These calculations of Barcel\'o and Carbery showed that in general the magnitude does not satisfy an inclusion-exclusion principle.  More recent work of Gimperlein and Goffeng~\cite{GimperleinGoffeng:MagnitudeFunction} showed that domains in a fixed Euclidean space  \emph{asymptotically} satisfy a version of the inclusion-exclusion principle.

Barcel\'o and Carbery essentially solved a differential equation to find the magnitude; we shall essentially solve an integral equation.  A key ingredient will be what is basically a special function identity, which we shall prove in Section~\ref{section:KeyIntegral} using PDE methods inspired by~\cite{BarceloCarbery}.

\subsection{The main results}
There are three main goals of this paper: firstly to show that the magnitude of an odd-dimensional ball can be calculated by solving the weight equation, without necessarily finding a potential; secondly, to give simple and explicit formulae for the numerator and denominator of the magnitude $\bigl|B^{2p+1}_R\bigr|$ in terms of `Hankel determinants' of the sequence $\left(\chi_i(R)\right)_{i=0}^\infty$ of `reverse Bessel polynomials'; and thirdly to give combinatorial expressions for these Hankel determinants and thus prove properties of the magnitude of odd dimensional balls.
Before stating the main results precisely we should define the reverse Bessel polynomials and the notion of Hankel determinant.

Let $\left(\chi_i(R)\right)_{i=0}^\infty$ denote the sequence of \define{reverse Bessel polynomials}, so that $\chi_i(R)$ is a degree $i$ integer polynomial in $R$.  The sequence begins as follows:
\begin{align*}
\chi_0(R)&=1;\\
\chi_1(R)&=R;\\
\chi_2(R)&=R^2+R;\\
\chi_3(R)&=R^3+3R^2+3R;\\
\chi_4(R)&=R^{4} + 6 R^{3} + 15 R^{2} + 15 R.
%\chi_5(R)&=R^{5} + 10 R^{4} + 45 R^{3} + 105 R^{2} + 105 R
\end{align*}
There are many ways to define this sequence, but we can take the recursion relation as a definition:
\begin{equation}
  \label{eq:ChiRecursion}
  \chi_{i+2}(R)=R^2\chi_i(R)+(2i+1)\chi_{i+1}(R).
\end{equation}
A standard reference for these polynomials is Grosswald~\cite{Grosswald:BesselPolynomials}, but Wikipedia~\cite{wikipedia:BesselPolynomial} is as good a place as any to learn about them;   note that the indexing used can vary from author to author: here, we have taken the version which forms a Sheffer sequence.  These reverse Bessel polynomials are related to the functions $\psi_i\colon(0,\infty)\to \R$ of Barcel\'o and Carbery~\cite{BarceloCarbery} by $\chi_i(R)=e^RR^{2i}\psi_i(R)$ --- see Proposition~\ref{prop:chiAndpsi} --- and are related to the modified spherical Bessel functions~\cite[Section~10]{NIST:DLMF} by $\chi_i(R)=\frac{2}{\pi} e^R R^{i+1}k_{i-1}(R)$.

Now let us see what a Hankel determinant is.  If $(\alpha_i)_{i=0}^\infty$ is a sequence of elements in a commutative ring, then, for $p=0,1,2,\dots$, the $p$th \define{Hankel determinant} of the sequence is the determinant
 \(\det[\alpha_{i+j}]_{i,j=0}^p\), that is the determinant of the matrix with constant anti-diagonals:
\[
\begin{vmatrix}
\alpha_0&\alpha_1&\alpha_2&\dots&\alpha_p\\
\alpha_1&\alpha_2&\dots&\dots&\alpha_{p+1}\\
\alpha_2&\dots&&\dots&\alpha_{p+2}\\
\vdots&&&&\vdots\\
\alpha_p&\dots&&\dots&\alpha_{2p}
\end{vmatrix}.
\]

%The goal of this paper is to use a different approach to Carbery and Barcel\'o, namely solving the weight equations directly without using potentials, to find the magnitude and to prove the following theorem giving  explicit expressions for the magnitude in terms of `Hankel' determinants of the sequence of reverse Bessel polynomials $\{\chi_i\}_{i=0}^\infty$ which will be defined below.
The main results on magnitude that will be proved in this paper can be collected into the following statement.
\begin{mainthm}%\label{thm:main} 
%For $n=2p+1$ an odd, positive integer, there are monic integer polynomials $N_p(R)$ and $D_p(R)$ with positive coefficients such that for $R>0$, the magnitude of the radius $R$ ball of dimension $n$ is given by the following ratio:
%\[\left|B^n_R\right|
%=\frac{N_p(R)}
%{n!\,D_p(R)}.
%\]
%These polynomials have the following degrees:
%\[\deg(N_p(R))= (p+1)(p+2)/2;\quad \deg(N_p(R))= (p-1)p/2.\]
%The constant terms are non-zero and related by
%\[N_p(0)=n!\,D_p(0).\]
%
Suppose that $n=2p+1$ is an odd positive integer.  The magnitude $\left|B^n_R\right|$ of the $n$-ball with radius $R$ can be expressed as the ratio of integer polynomials in the following way:
\[\left|B^n_R\right|
=\frac{\det[\chi_{i+j+2}(R)]^{p}_{i,j=0}}
{n!\,R\det[\chi_{i+j}(R)]^{p}_{i,j=0}},
\]
where $\chi_i(R)$ is the $i$th reverse Bessel polynomial.

There is a common monomial factor  in the above numerator and denominator.  Defining the $k$th superfactorial via $\super(k):=\prod_{i=0}^k i!$, we can write 
\(\left|B^n_R\right|
=\frac{N_p(R)}
{n!\,D_p(R)},
\)
where
\[N_p(R):= \frac{\det[\chi_{i+j+2}(R)]^{p}_{i,j=0}}{\super(p)R^{p+1}};
\qquad 
D_p(R):= \frac{\det[\chi_{i+j}(R)]^{p}_{i,j=0}}{\super(p)R^{p}}.\]

These polynomials have combinatorial expressions as sums of weights of collections of disjoint Schr\"oder paths (the terminology is explained in Section~\ref{section:SchroderPaths}):
\[
N_p(R)=\sum_{\sigma\in X_{p+1}}W_2(\sigma, R),\qquad
D_p(R)=\sum_{\sigma\in X_{p-1}}W_0(\sigma, R),
\]
where $X_k$ is a certain set of collections of disjoint Schr\"oder paths.

Both $N_p(R)$ and $D_p(R)$ are monic integer polynomials with positive coefficients and, in particular, non-zero constant terms. The constant terms satisfy $N_p(0)=n!\,D_p(0)$.  Thus $|B^n_R|\to 1$ as $R\to 0$.

The degrees of the polynomials are given by the following:
\[\deg(N_p(R))= (p+1)(p+2)/2;\qquad \deg(D_p(R))= (p-1)p/2.\]
The leading terms are as follows, for $\tempdim := \frac{1}{2}(p+1)(p+2)$:
\begin{align*}
  N_p(R)&= R^{\tempdim} +\tfrac{(p+1)^2(p+2)}{2} R^{\tempdim-1} + \tfrac{p(p+1)^3(p+2)(p+3)}{8} R^{\tempdim-2}+O(R^{\tempdim-3}),\\
  D_p(R)&=
  R^{\tempdim-2p-1} +\tfrac{(p-1)p(p+1)}{2} R^{\tempdim-2p-2} + \tfrac{(p-2)(p-1)p(p+1)^3}{8} R^{\tempdim-2p-3}+O(R^{\tempdim-2p-4}).
\end{align*}
Thus there is the following asymptotic behaviour of magnitude:
\[
  \left|B_R^n\right|=\frac{1}{n!}\left(R^n + \tfrac{n(n+1)}{2}R^{n-1}+ \tfrac{(n-1)n (n+1)^2}{8}R^{n-2}\right) + O(R^{n-3})
  \quad\text{as } R\to \infty.
\]
\end{mainthm}

\begin{proof}[Sketch proof.]  The rest of the paper will flesh out the following strategy.
\begin{enumerate}
\item Guess the form of the weighting distribution on the ball $B_R^n$.
\item Give an expression for   $\int_{S_R^{2p}}
    e^{-\left|\vec{x}-\vec{s}\right|}\,\mathrm{d}\vec{x}$ in terms of reverse Bessel polynomials.
\item Use the above two items to solve the weight equations for the ball.
\item Use Cramer's Rule to express the magnitude as a ratio of determinants.
\item Show that these determinants are equal to the Hankel determinants above.
\item Use the theory of continued fractions to give combinatorial expressions for the Hankel determinants in terms of counting weighted Schr\"oder paths.
\item Use these combinatorial expressions to give the degrees of the polynomials, the positivity of the coefficients and the formulae for the leading and constant terms. \qedhere
\end{enumerate}
\end{proof}
You are invited to perform the exercise of calculating the Hankel determinants in the theorem when $p=1$ and $p=2$ and verifying that the correct formulae for the magnitude are obtained.  (Many of the calculations in this paper are verified in the SageMath notebook available as~\cite{Willerton:SageWorksheet}.)

By way of comparison with the bound for the degree of the denominator of $\frac{3n^2-2n+7}{8}$  from~\cite{BarceloCarbery} mentioned above,  note that we have proved  that the degree of the denominator is bounded by $\frac{n^2-4n+3}{8}$.  This  bound is strict in dimensions that have been calculated numerically, i.e.~up to $n=45$.   This strictness is equivalent, of course, to the numerator and denominator having no common factor in those dimensions: in fact, in those dimensions the numerator and denominator polynomials are, according to SageMath, all irreducible over~$\mathbb{Q}$.

%Note a couple of things here.
%\begin{itemize}
%%\item The numerator and denominator in the above theorem are not quite as good as I would hope, yet, as I believe that both determinants become monic integer polynomials when divided by the superfactorial $\super(p)=\prod_{k=0}^p k!$.
%\item The numerator and denominator as given in the above theorem are irreducible up to $p=20$ according to SageMath calculations.
%\item It is not clear why the magnitude has such a simple, symmetric expression.  It is explained why it can be expressed as a ratio of determinants, namely that the magnitude can be shown to be one of the solutions to a set of simultaneous linear equations, thus has an expression as a ratio of determinants by Cramer's Rule.  However, some empirical coincidences suggested the form in terms of Hankel determinants and this is proved by uninformative row and column operations.
%\end{itemize}
From the above theorem we get the following corollary, which Barcel\'o and Carbery proved using Fourier theory~\cite[Theorem~1]{BarceloCarbery}.  (This does not necessarily hold for metric spaces that cannot be isometrically embedded in Euclidean space~\cite[Example~2.2.8]{Leinster:Magnitude}.)
\begin{cor}
If $X$ is a non-empty, compact subset of some Euclidean space and $tX$ denotes $X$ scaled by a factor of $t$ then
\[
  \left|tX\right| \to 1\quad\text{as }t\to 0.
\]
\end{cor}
\begin{proof}
As $X$ is non-empty, we have some $x\in X$.  As $X\subset\R^e$ for some $e$ we can isometrically embed $X$ in $\R^n$ for some odd $n$.  As $X$ is compact, it is bounded, so there is some $R$ with $X\subset B^n_R$.  Thus $\{tx\}\subset tX\subset tB^n_R=B^n_{tR}$.  By the monotonicity of magnitude for subsets of Euclidean space~\cite[Corollary~2.4.4]{Leinster:Magnitude} this means $1=\bigl|\{tx\}\bigr|\le \bigl|tX\bigr| \le \bigl|B^n_{tR}\bigr|$.  By the theorem above we have $\bigl|B^n_{tR}\bigr|\to 1$ as $t\to 0$ and the result follows from the Sandwich Rule.
\end{proof}

The rest of this introduction will describe how the Hankel determinant formula is obtained.   The details of proofs will be given in the body of the paper.  The combinatorial interpretation of the Hankel determinants and the remaining results about the magnitude are proved in 
Section~\ref{section:SchroderPaths}.

%The theorem is not supposed to be in any sense obvious, and  this paper does not make it much more obvious, in the sense that

%We are interested in odd dimensional balls so throughout this paper we will take $n\ge 3$ to be an odd integer with $n = 2p+1$, so $p\ge 1$ is an integer.  We let $R>0$ and denote by $B^n_R$ the $n$-dimensional ball of radius $R$.
\subsection{Weightings and weight distributions}
\label{sec:Weightings}
In this section we will give the definition due to Meckes of the correct notion of weighting and weight equation on a convex subset of Euclidean space.  The strategy for finding the magnitude of an odd dimensional ball will be to `simply' solve the weight equation.  We will start with a reminder of the notion of weighting and weight equation on a finite metric space.

Magnitude was defined by Leinster on finite metric space as follows.  If $A$ is a finite metric space then a \define{weighting} on $A$ is a function $w\colon A\to \R$ such that
\[
 \sum _{x} w(x) e^{-\mathrm{d}(x,s)} =1\quad \text {for all }s\in A.
\]
This is called the \define{weight equation}.
If a weighting exists then the \define{magnitude} is defined to be the total weight, $\left |A\right | \coloneqq \sum_{x} w(x)$. This is independent of any choice of weighting.  A weighting is known to exist for every finite subset of Euclidean space.

Following calculations in \cite{LeinsterWillerton:AsymptoticMagnitude}, Meckes~\cite{Meckes:PositiveDefinite} showed that magnitude could be extended to a unique maximal lower semicontinuous function on the space of  compact `positive definite' spaces.  This class of metric spaces includes compact subsets of Euclidean space.   One way %~\cite{Meckes:PositiveDefinite} 
of calculating the magnitude of an infinite compact subset $X\subset \R^n$ of Euclidean space is by taking a sequence $(A_i)_{i=0}^\infty$ of finite subsets --- $A_i\subset X$ --- which tend to $X$ in the Hausdorff topology, then we have $|X|=\lim_i |A_i|$.

Taking a weighting on  $A_i$ for each $i$, gives a sequence of finite signed measures on $X$.  One might hope that they tend to a measure $\mu$ on $X$ which satisfies the obvious analogue of the weighting equation above, namely
\[
 \int_{x}  e^{-\mathrm{d}(x,s)} \mathrm{d}\mu(x)=1\quad \text {for all }s\in X,
\]
and that one could then calculate the magnitude $|X|$ as the total mass of $\mu$, ie.~$\int_X \mathrm{d}\mu$.  This will work for some spaces such as the closed interval~$B^1_R$, see~\cite{Willerton:SpheresSurfaces}.  Unfortunately, it does not work in general, as a sequence of finite signed measures which gives a convergent sequence when evaluated on any function does not  typically give rise to a signed measure.  Here is a simple example of that.

Define the sequence $(\mu _{i})_{i=0}^{\infty }$ of finitely supported signed measures on the real line $\mathbb{R}$ by $\mu _{i}\coloneqq i\delta _{0} - i\delta _{1/i}$, where $\delta_a$ is the Dirac delta measure supported at $a\in \R$. Then for any differentiable function $f$ we get the following convergence:
\[
\int _{\mathbb{R}} f \,\dd\mu _{i} = \frac{f(0)-f(1/i)}{1/i} \to f'(0)\qquad \text {as}\qquad i\to \infty .
\]
 But the functional $f\mapsto f'(0)$ is not represented by any signed measure. Rather it is a distribution.  
 
 We will need to use distributions as our limiting object of finite signed measures, and so use weight \emph{distributions} to generalize weight measures.
In general terms, a distribution on $\mathbb{R}^{n}$ is a linear functional on some suitable class of functions, say smooth and decaying appropriately to zero at $\infty $. We will describe some specific spaces of distributions below.  We will write $\langle w, f\rangle $ for the evaluation of a distribution $w$ on a function $f$.

Here are a few of examples of distributions.
\begin{enumerate}[(i)]
\item For each (appropriately integrable) function $g$ we have an associated distribution with $\langle g, f\rangle \coloneqq \int _{\vec{x}\in\mathbb{R}^{n}} g(\vec{x})f(\vec{x})\, \dd\vec{x}$.
\item For each signed measure $\mu $ we have an associated distribution with $\langle \mu , f\rangle \coloneqq \int _{\mathbb{R}^{n}} f \,\dd \mu $.
\item Generalizing the derivative mentioned above, for any cooriented, smooth, codimension one submanifold $\Sigma$ of $\mathbb{R}^{n}$, and $i\in \N$ we have the distribution $w_{i}$ given by
  \[
  \langle w_i, f\rangle \coloneqq \int _{\Sigma} \dnormali{i} f(\vec{x}) \,\dd\vec{x},
  \]
   where $\frac{\partial }{\partial \nu }$ means derivative in the normal direction to the submanifold.
\end{enumerate}

%We thus define a \define{weight distribution} to be a distribution $w$ such that we have
%\[
%\langle w, e^{-d(\cdot , y)}\rangle =1\quad \text {for all } y\in X.
%\]
% Mark showed that every compact subset of Euclidean space has a weight distribution and that the magnitude of such a subset is given by
%\[
%|X| =\langle w, 1\rangle
%\]
% where $1$ represents any function which is identically equal to $1$ on $X$.
%
%Mark showed that having a weight distribution corresponds to having a 'potential function'. I will look at what such a thing is for the example of the $1$-ball.

Meckes~\cite{Meckes:MagnitudeDimensions} showed that, for the magnitude of subsets of $\R^n$,  we need to consider distributions in the following \define{Bessel potential space}:
\[H^{-i}(\mathbb{R}^{n}) := 
 \left\{w \in \mathcal{S}'(\mathbb{R}^{n})\mid  (1+{\left\| \cdot \right\|}^2)^{-i/2}\hat w\in L^2(\mathbb{R}^{n})\right\},\]
where $\mathcal{S}'(\mathbb{R}^{n})$ is the space of tempered distributions and $\hat w$ is the Fourier transform of the distribution $w$.  

For non-analysts the definition of Bessel potential space might look a little intimidating, but we will not need to know what tempered distributions or the Fourier transform of a distribution are, and will really only need the theorem below together with the fact that certain specific distributions, similar to the examples above, live in a Bessel potential space.  These are all proved in Section~\ref{section:BesselPotential}.

\begin{thm}[Meckes~\cite{Meckes:PersComm}]
\label{thm:MagnitudeWeighting}
  Suppose $K$ is a compact, convex subset of $\R^n$, for $n=2p+1$,  with nonempty
  interior, $\interior(K)$.   Suppose that we have a distribution $w\in H^{-(p+1)}(\R^n)$, which is supported on K and
  satisfies 
  \[\langle w, e^{-\mathrm{d}(\vec{s},\cdot)}\rangle = 1 \quad\text{for every }\vec{s} \in \interior(K),\] 
  then $w$ is the unique such distribution on $K$.  
  
  Moreover, if $\one$ is any smooth function which is $1$ on $K$ %and sufficiently rapidly decaying 
then the magnitude of $K$ is given by
  \begin{equation}
  \left|K\right| = \langle w, \one\rangle.
  \label{eqn:magdefn}
  \end{equation}
\end{thm}
\noindent This theorem is essentially a summary of ideas from~\cite{Meckes:MagnitudeDimensions} and~\cite{BarceloCarbery}.  Note that in the above $\langle w, e^{-\mathrm{d}(\vec{s},\cdot)}\rangle$ is well defined because, although the function $e^{-\mathrm{d}(\vec{s},\cdot)}$ is not differentiable at the origin, it does lie in $H^{(p+1)}(\R^n)$.

 We will call a distribution which satisfies the hypotheses of the above theorem a \define{weight distribution}.

\subsection{How we find the magnitude of an odd ball}
We will find the magnitude of an odd ball by using Theorem~\ref{thm:MagnitudeWeighting} above.  So we want to find a weight distribution $w$ on the $n$-ball $\ball \subset \R^n$, this means a distribution $w\in H^{-(p+1)}(\R^n)$  which is supported on the ball   and satisfies the weight equation
\begin{equation}
  \left\langle w, e^{-\left|{\cdot}-\vec{s}\right|}\right\rangle =1\qquad\text{for all }\vec{s}\in \R^n\text{ with }|\vec{s}|<R.
%  \tag{$*$}
  \label{eqn:weightdistribution}
\end{equation}
%Given such a weight distribution, the magnitude is given by evaluating the distribution on the constant unit function:
%\begin{equation}
 % \left|\ball \right|=\langle w, 1\rangle.
 % \tag{$\dagger$}
%  \label{eqn:magdefn}
%\end{equation}
 We will guess that the weight distribution $w$ has a particular form, namely some multiple of the Lebesgue measure on the ball plus some linear combination of integration of normal derivatives on the boundary sphere;  so we will guess that for a suitably smooth function $f$ the distribution evaluates on $f$ as follows:
 \begin{equation}
   \langle w, f\rangle
   =
   \frac{1}{n!\,\omega_n}\left(
   \int_{\vec{x}\in\ball} f\,\dd\vec{x} +
  \sum_{i=0}^p 
    \beta_i(R) \int_{\vec{x}\in\sphere} \dnormali{i} f\,\dd\vec{x}
    \right),
%   \beta_0(R) \int_{\sphere} f
%   +
%   \beta_1(R) \int_{\sphere} \dnormal f
%   + {}
%   \right.
%   \\
%   \left.
%   \dots+\beta_p(R) \int_{\sphere} \dnormali{p} f
%   \right),
%   \tag{$**$}
   \label{eqn:defineweighting}
\end{equation}
where, as above, $\dnormal$ means differentiate in the normal direction to the boundary, where $\omega_n$ is the volume of the unit $n$-dimensional ball and where $\left\{\beta_i(R)\right\}_{i=0}^p$ is a set of unknown functions of $R$ which will be found by solving the weight equation~\eqref{eqn:weightdistribution}.  This guess is based on low-dimensional calculations following on from ideas in~\cite{BarceloCarbery}.  The fact that this distribution is in the Bessel potential space $H^{-(p+1)}(\R^n)$ is proved in Section~\ref{section:ExamplesInBesselSpace}.

Once we have shown that we can solve the weight equation for this $w$ then we can find the magnitude using equation~\eqref{eqn:magdefn}:
\begin{align}
  \left|\ball \right|
  &=  \langle w, \one\rangle
  =
  \frac{1}{n!\,\omega_n}\left(
   \int_{\ball} \one \,\dd\vec{x}+
   \beta_0(R) \int_{\sphere} \one\,\dd\vec{x}\right)\notag\\
   &=
  \frac{1}{n!\,\omega_n}\left(
  \omega_n R^n +\beta_0(R) \sigma_{n-1}R^{n-1}
   \right)\notag \\
   &=
  \frac{1}{n!}\left(
   R^n +n\beta_0(R) R^{n-1}
   \right),
%   \tag{$\dagger\dagger$}
   \label{eqn:magbetazero}
\end{align}
where $\sigma_{n-1}$ denotes the volume of the unit $(n-1)$-sphere and where we have used the fact that $\sigma_{n-1}/\omega_n = n$.  So the goal now is to solve the weight equation~\eqref{eqn:weightdistribution} using a weight distribution of the form \eqref{eqn:defineweighting}; to do this we will need to prove a rather non-obvious integral identity.

\subsection{The key integral identity and its generalization}
The key to the solution to the weight equation \eqref{eqn:weightdistribution} lies in evaluating, for $\vec{s}\in\interior(\ball)$, the integral over the sphere
\[
  \frac{1}{n!\,\omega_n}
  \int_{\vec{x}\in\sphere}
  e^{-\left|\vec{x}-\vec{s}\right|}\,\mathrm{d}\vec{x},
\]
where, as usual, $\omega_n$ is the volume of the unit $n$-ball.
By spherical symmetry, the only dependence of this expression on the vector $\vec{s}$ is via its magnitude $|\vec{s}|$: we will write $s=|\vec{s}|$.  This expression will be thought of as a function of $R$ and $s$, with $R>s\ge 0$ and we are going to express it in terms of the reverse Bessel polynomials and another sequence of  functions which we will now introduce.

%Let $(\chi_i)_{i=0}^\infty$ denote the sequence of \define{reverse Bessel polynomials}, so that $\chi_i(R)$ is a degree $i$ integer polynomial in $R$.  The sequence begins as follows:
%\begin{align*}
%\chi_0(R)&=1;\\
%\chi_1(R)&=R;\\
%\chi_2(R)&=R^2+R;\\
%\chi_3(R)&=R^3+3R^2+3R;\\
%\chi_4(R)&=R^{4} + 6 R^{3} + 15 R^{2} + 15 R.
%%\chi_5(R)&=R^{5} + 10 R^{4} + 45 R^{3} + 105 R^{2} + 105 R
%\end{align*}
%There are many ways to define this sequence, but we can take the recursion relation as a definition:
%\[
%  \chi_{i+2}(R)=R^2\chi_i(R)+(2i+1)\chi_{i+1}(R).
%\]

The reverse Bessel polynomials are closely related to the modified spherical Bessel functions of the second kind.  The next sequence of functions we are interested in are closely related to modified spherical Bessel functions of the \emph{first} kind~\cite[Section 10]{NIST:DLMF}, but I don't have a good name for them.  Let  $(\psit_i)_{i=0}^\infty$ denote the sequence of functions $\R\to \R$ which has $\psit_0=\cosh$ and satisfies the recursion relation
\[
  \psit_{i+1}(s)=-\tfrac{1}{s}\psit_i'(s).%\frac{\mathrm d \psit_i(s)}{\mathrm{d} s}.
\]
%This is one of many ways to define the sequence. %; integral expressions are given in Section~\ref{section:PlausibleStrategy}.
The sequence
begins in the following way:
 \begin{align*}
\psit_0(s)&=\cosh(s);\\
\psit_1(s)
&=-\frac{\sinh\left(s\right)}{s};\\
\psit_2(s)
&=\frac{\cosh\left(s\right)}{s^{2}} - \frac{\sinh\left(s\right)}{s^{3}};\\
\psit_3(s)
&=
-\frac{\sinh\left(s\right)}{s^{3}} + \frac{3 \, \cosh\left(s\right)}{s^{4}} - \frac{3 \, \sinh\left(s\right)}{s^{5}};\\
\psit_4(s)
&=
\frac{\cosh\left(s\right)}{s^{4}} - \frac{6 \, \sinh\left(s\right)}{s^{5}} + \frac{15 \, \cosh\left(s\right)}{s^{6}} - \frac{15 \, \sinh\left(s\right)}{s^{7}}.
\end{align*}
It is not obvious that these functions are well defined at $s=0$, but this is proved in Proposition~\ref{prop:psiTilde0}.  %The notation comes from the fact that $\psit_i$ is related to the function $\psi_i \colon \R\backslash\{0\}\to \R$ from \cite{BarceloCarbery} by $\psit(s)=\tfrac12(\psi(s)+\psi(-s))$ for $s\ne 0$.  This will be examined further in Section~\ref{section:KeyIntegral}.
%Note that the coefficients are the same as for the reverse Bessel polynomials, as $\psi_i(s)=e^{-s}s^{-2i}\chi_i(s)$, but this does not appear to be relevant in this paper.

We can now state the key result required for solving the weight equation.
\begin{thm}[The Key Integral]
\label{thm:KeyIntegral}
  For $n=2p+1$ an odd integer, $R>0$, with $\vec{s}$ a point in the interior of the ball $\ball$, and $s=|\vec{s}|$, then
  \[
    \frac{1}{n!\,\omega_n}
    \int_{\vec{x}\in\sphere}
    e^{-\left|\vec{x}-\vec{s}\right|}\,\mathrm{d}\vec{x}
    =
    \frac{(-1)^p e^{-R}}{2^p p!}\sum_{i=0}^p \binom{p}{i}\chi_{p+i}(R)\psit_i(s).
  \]
\end{thm}
%I imagine that this would probably be a bread-and-butter kind of result for a nineteenth century mathematician, it can be phrased as an integral identity for modified spherical Bessel functions, and I will do that in another paper.  %I don't yet have a proof of this, but have a plausible looking strategy using explicit expressions for the integral over the sphere, see Section~\ref{section:PlausibleStrategy}.
%The conjecture has been checked up to $p=18$ on Sage.

%\subsection{Aside on Bessel functions}
In the process of proving the above key integral we will prove a more general result. We will need to introduce the sequence of functions $(\psi_i(r))_{i=0}^\infty$, used by Barcel\'o and Carbery~\cite{BarceloCarbery}.  These can be defined by taking $\psi_0(r)=\e^{-r}$ and using the same recurrence relation as for $\psit_i$, namely,
\[
  \psi_{i+1}(r)=-\tfrac{1}{r}\psi_i'(r).%\frac{\mathrm d \psit_i(s)}{\mathrm{d} s}.
\]
%This is one of many ways to define the sequence. %; integral expressions are given in Section~\ref{section:PlausibleStrategy}.
The sequence
begins in the following way:
 \begin{align*}
\psi_0 (r) &= \e^{-r}\\[0.5em]
\psi_1 (r) &= \e^{-r} \left( \frac{1}{r} \right)\\
\psi_2 (r) &= \e^{-r} \left( \frac{1}{r^{2}} + \frac{1}{r^{3}} \right)\\
\psi_3 (r) &= \e^{-r} \left( \frac{1}{r^{3}} + \frac{3}{r^{4}} + \frac{3}{r^{5}} \right)\\
\psi_4 (r) &= \e^{-r} \left( \frac{1}{r^{4}} + \frac{6}{r^{5}} + \frac{15}{r^{6}} + \frac{15}{r^{7}} \right)
.
\end{align*}
These are related to modified spherical Bessel functions of the second kind (see below) and are related to the reverse Bessel polynomials by $\chi_i(r)=r^{2i}\e^r\psi_i(r)$ as shall be proved in Proposition~\ref{prop:chiAndpsi}. 

We can now state the more general theorem which we shall prove.
\begin{thm}
\label{thm:GeneralKeyIntegral}
For $0\le j \le p$, $\vec{s}\in \R^{2p+1}$, with $s=|\vec{s}|$ and  $R>s\ge 0$
\begin{equation*}
  \int_{\vect{x}\in S^{2p}_R} \psi_j(\left | \vect{x}-\vect{s}\right|)\,\dd\vect{x}
  =
  (-2\pi)^p  2 \e^{-R}\sum_{i=0}^{p-j} \binom{p-j}{i}\chi_{i+p}(R) \psit_{i+j}(s).
  %\label{eq:GeneralizedKeyIntegral}
\end{equation*}
\end{thm}
\noindent The integral we want, Theorem~\ref{thm:KeyIntegral}, is then just the case $j=0$, as the volume of the unit radius odd ball has the formula $\omega_{2p+1}=\frac{2(p!) (4\pi)^p}{(2p+1)!}$.  We will prove this more general case, because we obtain the case $j=0$ by proceeding inductively downward from $j=p$.  The proof of this theorem is given in Section~\ref{section:KeyIntegral} and uses PDE methods inspired by results in~\cite{BarceloCarbery}.

It is worth noting here the relationship of the above integral with a Bessel function identity.  One can show that our functions are related to the modified spherical Bessel functions via 
\[
\psit_i(x)=(-1)^i\sqrt{\frac{\pi}{2}}\frac{I_{i-1/2}(x)}{x^{i-1/2}}
\quad\text{and}\quad
\psi_i(x)=\sqrt{\frac{2}{\pi}}\frac{K_{i-1/2}(x)}{x^{i-1/2}}
,\]
where $I_\nu$ and $K_\nu$ are the modified Bessel functions of the first and second kind respectively.
 Using these and polar coordinates on the sphere, we can write the above integral identity in the following way,
\begin{align*}
\sqrt{\frac{2}{\pi}}&\int_0^\pi \frac{K_{j-1/2}(w)}{w^{j-1/2}}\sin^{2p-1}(\theta)\, \dd\theta \\
&= (-1)^{p-j}(p-1)!\,  2^p %\\&\qquad
\sum_{m=0}^{p-j}(-1)^m \binom{p-j}{m} 
              R^{2m} \frac{K_{p+m-1/2} (R)}{R^{p+m-1/2}}\cdot
              \frac{I_{m+j-1/2} (s)} { s^{m+j-1/2}},
\end{align*}
where $w=\sqrt{R^2+s^2-2Rs\cos\theta}$.  

I have not been able to find this result in the literature at all.  However, the `base case' for the induction, when $j=p$, becomes the following
\[
\sqrt{\frac{2}{\pi}}\int_0^\pi \frac{K_{p-1/2}(w)}{w^{p-1/2}}\sin^{2p-1}(\theta)\, \dd\theta 
= (p-1)! \, 2^p %\\&\qquad
             \frac{K_{p-1/2} (R)}{R^{p-1/2}}\cdot
              \frac{I_{p-1/2} (s)} { s^{p-1/2}},
\]
and this --- or the unmodified analogue --- is standard in Bessel function literature, e.g.~\cite[11.42~(16)]{Watson:BesselFunctions} or \cite[(4.11.6)]{Andrews:SpecialFunctions}.
%, but I could not find the equivalent of the general expression, equation~\eqref{eq:GeneralizedKeyIntegral}, in the literature.  More details on this will appear elsewhere.

\subsection{Corollaries of the key integral identity}
In order to solve the weight equation \eqref{eqn:weightdistribution} using the weighting $w$ defined in \eqref{eqn:defineweighting}, as well as integrating the expression $ e^{-\left|\vec{x}-\vec{s}\right|}$ over the sphere, as in the key result above, we also need to integrate it over the ball, and to integrate its normal derivatives over the sphere.  Fortunately, these integrals can be obtained reasonably straightforwardly from the key result.  The proofs are given in Section~\ref{section:ProofsOfCorollaries}.

Firstly, integrating over the ball we have the following.
\begin{cor}
\label{cor:FirstCor}
    For $n=2p+1$ an odd integer, $R>0$, $\vec{s}$ a point in the interior of the ball $\ball$, and $s=|\vec{s}|$, then
  \[
    \frac{1}{n!\,\omega_n}
    \int_{\vec{x}\in\ball}
    e^{-\left|\vec{x}-\vec{s}\right|}\,\mathrm{d}\vec{x}
    =
    1 -
    \frac{(-1)^p e^{-R}}{2^p p!}\sum_{i=0}^p \binom{p}{i}\frac{\chi_{p+i+1}(R)}{R}\psit_i(s).
  \]
\end{cor}
\noindent This is proved in Section~\ref{section:ProofOfFirstCor}.

For integrating the normal derivatives over the sphere, it makes sense to simplify the notation by introducing the operator $\deltap$ defined on functions of $R$ by
\[
  \deltap f(R):=
  e^{-R}R^{2p}\frac{\mathrm d}{\mathrm d R}\left(e^R R^{-2p}f(R)\right).\]
Remember that $p$ is supposed to be a fixed integer, so although $\deltap$ depends on $p$ we won't include it in the notation.
\begin{cor}
\label{cor:SecondCor}
  For $n=2p+1$ an odd integer, $R>0$, $\vec{s}$ a point in the interior of the ball $\ball$, and $s=|\vec{s}|$, then
  \[
    \frac{1}{n!\,\omega_n}
    \int_{\vec{x}\in\sphere}
    \dnormali{j}e^{-\left|\vec{x}-\vec{s}\right|}\,\mathrm{d}\vec{x}
    =
    \frac{(-1)^p  e^{-R}}{2^p p!}\sum_{i=0}^p \binom{p}{i}
    \deltap^j\chi_{p+i}(R)\psit_i(s).
  \]
\end{cor}
\noindent This is proved in Section~\ref{section:ProofOfSecondCor}.

We will see below in Lemma~\ref{lemma:delta} that $\deltap^j\chi_{p+i}(R)$ can be expressed as a suitable linear combination of reverse Bessel functions.
\subsection{Solving the weight equation}  
We can now use the corollaries of the key integral identity to solve the weight equation for an odd ball. The weight equation \eqref{eqn:weightdistribution} for the weighting $w$ defined in~\eqref{eqn:defineweighting} is the following:
\begin{equation*}
   \frac{1}{n!\,\omega_n}\left(
   \int_{\vec{x}\in\ball}     e^{-\left|\vec{x}-\vec{s}\right|}\,\mathrm{d}\vec{x} +
  \sum_{i=0}^p 
    \beta_i(R) \int_{\vec{x}\in\sphere} \dnormali{i}    e^{-\left|\vec{x}-\vec{s}\right|}\,\mathrm{d}\vec{x}
    \right)
    =
    1
    \quad
    \text{for all }
    |\vec{s}|<R.
\end{equation*}
We want to solve this to find the set $\{\beta_j(R)\}_{j=0}^p$.  Substituting in the expressions from Corollaries~\ref{cor:FirstCor} and~\ref{cor:SecondCor} we find
\begin{multline*}
  1 -
    \frac{(-1)^p e^{-R}}{2^p p!}\sum_{i=0}^p
    \binom{p}{i}\frac{\chi_{p+i+1}(R)}{R}\psit_i(s)
    +\sum_{j=0}^p \beta_j(R) \frac{(-1)^p e^{-R}}{2^p p!}\sum_{i=0}^p \binom{p}{i}
    \deltap^j\chi_{p+i}(R)\psit_i(s)
    =1.
\end{multline*}
Rearranging and simplifying leads to the following:
\[
  \sum_{i=0}^p \binom{p}{i}\sum_{j=0}^p \beta_j(R)
    \deltap^j\chi_{p+i}(R)\psit_i(s)
  =
  \sum_{i=0}^p
    \binom{p}{i}\frac{\chi_{p+i+1}(R)}{R}\psit_i(s).
\]
As this needs to be true for all $s<R$ and as the set of functions $\{\psit_i\}_{i=0}^\infty$ is linearly independent, we can compare the coefficients of $\psit_i(s)$ for each $i=0,\dots,p$ and obtain the following linear system:
\[
  \sum_{j=0}^p
    \deltap^j\chi_{p+i}(R)\beta_j(R)
  =
  \chi_{p+i+1}(R)/R
  \qquad
  \text{for }i=0,\dots,p.
\]
We have thus obtained the following theorem.
\begin{thm}\label{thm:linearsystem}
The set of functions $\{\beta_j(R)\}_{i=0}^p$ satisfies the linear system
\[
\begin{pmatrix}
\chi_p(R)&  \deltap \chi_p(R)& \deltap^2 \chi_p(R)&\dots&\deltap^p\chi_p(R)\\
\chi_{p+1}(R)&  \deltap \chi_{p+1}(R)& \deltap^2 \chi_{p+1}(R)&\dots&\deltap^p\chi_{p+1}(R)\\
\vdots&\vdots&\vdots&&\vdots\\
\chi_{2p}(R)&  \deltap \chi_{2p}(R)& \deltap^2 \chi_{2p}(R)&\dots&\deltap^p\chi_{2p}(R)\\
\end{pmatrix}
\begin{pmatrix}
\beta_0(R)\\
\beta_1(R)\\
\vdots\\
\beta_{p}(R)
\end{pmatrix}
=
\begin{pmatrix}
\chi_{p+1}(R)/R\\
\chi_{p+2}(R)/R\\
\vdots\\
\chi_{2p+1}(R)/R
\end{pmatrix}
\]
if and only if the distribution $w$ defined in~\eqref{eqn:defineweighting} is a weight distribution on the ball $B^n_R$.
\end{thm}
Solving this linear system will lead to expression for the magnitude, as from~\eqref{eqn:magbetazero} we have that the magnitude is given by $\left|\ball \right|=
  \frac{1}{n!}\left(
   R^n +n\beta_0(R) R^{n-1}
   \right)$.
We will see below how to use this to get a formula for the magnitude without explicitly finding $\beta_0(R)$, but first it is worth pausing for an example.
\subsection{An example}
It is very easy from the above theorem to get a computer algebra system such as SageMath calculate the entries in the linear system and solve it (for instance, see~\cite{Willerton:SageWorksheet}).  The entries are almost polynomials in the sense that multiplying through by $R^{p-1}$ will give all entries being integer polynomials which means that computation can be done quite efficiently.

For example, we can look at the case where $n=5$, i.e.~$p=2$.  There the system looks as follows.
\begin{multline*}
\biggl(\begin{smallmatrix}
R^{2} + R & -R^{2} - 3R - 3 & {(R^{3} + 5R^{2} + 12R +
12)}/{R} \\
R^{3} + 3R^{2} + 3R & -R^{3} - 4R^{2} - 9R - 9 & {(R^{4} +
5R^{3} + 17R^{2} + 36R + 36)}/{R} \\
R^{4} + 6R^{3} + 15R^{2} + 15R & -R^{4} - 6R^{3} - 21R^{2} - 45R -
45 &{(R^{5} + 6R^{4} + 27R^{3} + 87R^{2} + 180R + 180)}/{R}
\end{smallmatrix}\biggr)
%\scriptsize
%\times
\biggl(
\begin{smallmatrix}
\beta_0(R)\\
\beta_1(R)\\
\beta_{2}(R)
\end{smallmatrix}
\biggr)
\\=
\biggl(
\begin{smallmatrix}
R^{2} + 3R + 3 \\
R^{3} + 6R^{2} + 15R + 15 \\
R^{4} + 10R^{3} + 45R^{2} + 105R + 105
\end{smallmatrix}
\biggr)
\end{multline*}
And the solution is
\[
\begin{pmatrix}
\beta_0(R)\\
\beta_1(R)\\
\beta_{2}(R)
\end{pmatrix}
=
\frac{1}{(R + 3)R^4}
\begin{pmatrix}
{3R^5 + 27R^4 + 105R^3 + 216R^2 + 216R + 72}\\
 (3R^4 + 29R^3 + 108R^2 + 180R + 120)R\\
 (R^3 + 9R^2 + 27R +24)R^2
\end{pmatrix}.
\]
We are really interested in the magnitude and this can be calculated using the formula~\eqref{eqn:magbetazero}:
\begin{align*}
  \left|B_R^5 \right|
   &=
  \frac{1}{n!}\left(
   R^n +n\beta_0(R) R^{n-1}
   \right),
   \\
   &=
   \frac{R^{5}}{5!}+\frac{3R^5+27R^4+105R^3+216R^2+216R+72}{4!(R+3)}
   \\
   &=
    \frac{1}{5!} \frac{R^{6} + 18R^{5} + 135R^{4} + 525R^{3} + 1080R^{2} + 1080R + 360}{R + 3}.
\end{align*}
Reassuringly, this is the same as the formula given at the beginning of the paper.

 %This is all well and good, and it gives us a way to calculate the magnitude of any odd dimensional ball.  However, in the next section we will give charming, yet mysterious formulae for the numerator and denominator of the magnitude in terms of Hankel determinants of reverse Bessel polynomials.

\subsection{Hankel determinant formula}
We can now get to the formula for magnitude in terms of Hankel determinants of reverse Bessel polynomials.

We know from~\eqref{eqn:magbetazero} that the magnitude of an odd dimensional ball satisfies the following equation
\[
  -nR^{n-1}\beta_0 (R) +n!
  \left|\ball \right|
  =
  R^n.
\]
We can include this in the linear system from Theorem~\ref{thm:linearsystem} to get the extended linear system
\begin{equation}
\label{eq:ExtendedLinearSystem}
\begin{pmatrix}
\chi_p(R)& \deltap\chi_{p}(R)& \dots&\deltap^p\chi_p(R) &0\\
\chi_{p+1}(R)& \deltap\chi_{p+1}(R)& \dots&\deltap^p\chi_{p+1}(R)&0\\
\vdots&\vdots&&\vdots&\vdots\\
\chi_{2p}(R)& \deltap\chi_{2p}(R) &\dots&\deltap^p\chi_{2p}(R)&0\\
-nR^{n-1}&0 &\dots&0&n!
\end{pmatrix}
\begin{pmatrix}
\beta_0(R)\\
\beta_1(R)\\
\vdots\\
\beta_{p}(R)\\
\left|\ball \right|
\end{pmatrix}
=
\begin{pmatrix}
\chi_{p+1}(R)/R\\
\chi_{p+2}(R)/R\\
\vdots\\
\chi_{2p+1}(R)/R\\
R^n
\end{pmatrix}.
\end{equation}
Cramer's Rule then gives us a formula for the magnitude $\left|\ball \right|$ as a ratio of determinants.  Then by various properties of the reverse Bessel polynomials and sequences of row and column operations that are detailed in Section~\ref{section:HankelDets} we obtain the following theorem.

\begin{thm}\label{thm:hankel}  For $n=2p+1$ the magnitude of the $n$-dimensional radius~$R$ ball can be expressed as the following ratio involving Hankel determinants:
\[|B_R^n|=\frac{\det [\chi_{i+j+2}(R)]_{i,j=0}^p}{n!\,R\det[\chi_{i+j}(R)]_{i,j=0}^p}.\]
\end{thm}
%
%To complete Theorem~\ref{thm:main} it suffices to prove the following which is done in Section~\ref{section:DetCalcs}.
%\begin{thm} For $p\ge 0$ we have
%\begin{align*}
%\biggl[R^{-p}\det [\chi_{i+j}(R)]_{i,j=0}^p \biggr]_{R=0}&= \prod_{i=0}^p(2i+1)! %+O(R^{p+1})
%\\
%\biggl[ R^{-p-1}\det [\chi_{i+j+2}(R)]_{i,j=0}^p \biggr]_{R=0}&= \prod_{i=0}^p(2i+1)!\ .% +O(R^{p+2})
%\end{align*}
%\end{thm}

We have now achieved the second of our goals.  This formula is a very compact and beautiful way of expressing the magnitude of an odd ball.  The formula, however, remains somewhat mysterious.  It is not clear why such a formula should exist.  Certainly the heavily computational way it is obtained in Section~\ref{section:HankelDets} does not shed any light on this.  I was able to guess the formula as I spotted  patterns in some other calculations.

All that aside, this formula does give an easy and memorable way to implement the calculation of magnitude in computer algebra such as SageMath or Maple.  However,  as it stands it does not tell us much about the actual polynomials in the numerator or denominator.  For instance, I do not know how to show directly from the above definition what the degrees of the polynomials are or that the coefficients are all positive.  Fortunately, there are deep connections between Hankel determinants, continued fractions and combinatorics that can be exploited here.

\subsection{Combinatorial formulae via Schr\"oder paths}
There is a classical theory allowing the calculation of the Hankel determinants of a sequence of \emph{numbers} $(a_i)_{i=0}^{\infty}$ using a continued fraction expansion of the generating function $\sum_i a_i t^i$ of the sequence.  A stumbling block in using this approach in the current context is that this usually gives a simple factorization of the Hankel determinants whereas the Hankel determinants here turn out to typically be {irreducible} polynomials (up to trivial factors of $R$).  This basically means that the generating function of the reverse Bessel polynomials cannot have a continued fraction expansion of the usual Stieltjes-type or Jacobi-type.  There has been more recent work on the case when the sequence consists of polynomials rather than numbers; in particular, Alan Sokal has a large work in progress~\cite{Sokal:ContinuedFractionsBook}.   I asked Sokal about the reverse Bessel polynomials, and he was able to find a Thron-type continued fraction expansion (see Theorem~\ref{thm:ContinuedFraction}).  This made the reverse Bessel polynomials the first non-trivial example he had come across of a  sequence of polynomials which has a Thron expansion but no Stieltjes or 
Jacobi expansion.
 
 With this continued fraction expansion in hand, one can apply combinatorial techniques --- such as the Karlin-McGregor-Lindstr\"om-Gessel-Viennot Lemma (Theorem~\ref{thm:KMLGV}) --- to give formulae for the Hankel determinants appearing in the magnitude in terms of counting weighted `Schr\"oder paths'.   We see in Theorem~\ref{thm:CombinatorialDeterminant} that
 \begin{align*}
  \det\left[\chi_{i+j}(R)\right]_{i,j=0}^p
  &=
  R^p\super(p)\sum_{\sigma\in X_{p-1}}W_0(\sigma, R),
  \\
  \det\left[\chi_{i+j+2}(R)\right]_{i,j=0}^p
  &=
  R^{p+1}\super(p)\sum_{\sigma\in X_{p+1}}W_2(\sigma, R),
\end{align*}
where the sums are over certain collections of Schr\"oder paths.
In fact, each summand in the two sums above is a monomial so it is possible to give combinatorial expressions for each of the coefficients in the numerator and denominator polynomials.   These expressions are used in Section~\ref{section:SchroderPaths} to prove all of the remaining results about the magnitude of odd balls in the Main Theorem.

\subsection{Further observations about the magnitude of odd balls}
We finish the introduction with a few observations which merit additional study.

You might have noticed that the numerator and denominator at low degree are unimodal --- so the coefficients rise to a maximum and then fall --- or more specifically are log-concave --- a polynomial with positive coefficients $\sum_i a_i R^i$ is log-concave  if $a_i^2\ge a_{i-1} a_{i+1}$ for all $i$.  According to SageMath, both $N_p(R)$ and $D_p(R)$ are log-concave for $p\le 22$; however, I do not even know how to prove that they are unimodal.  Sokal~\cite{Sokal:ContinuedFractionsBook} has other examples of Hankel determinants of sequences of enumerative polynomials that appear, according to computer calculations, to be log-concave, so this is perhaps a more general phenomenon.

To give an example, in Figure~\ref{fig:LogConcavity}, the logarithms of the coefficients in the numerator polynomial $N_{22}(R)$ are plotted and they are seen to form a concave set of points.  Plotting the coefficients of the denominators gives similar looking results.  An alternative way to think of the polynomials is in terms of their roots.  The  roots of $N_{22}(R)$ as pictured in Figure~\ref{fig:LogConcavity} were calculated using high-precision arithmetic in SageMath~\cite{Willerton:SageWorksheet}.  The pattern is typical for all numerators and denominators that have been calculated (i.e.~up to $p=22$) in that the roots all lie in the sector $\{z\mid 3\pi/4< \arg(z)<5\pi/4\}$.
This gives added weight to the idea that the polynomials are log-concave as any real polynomial with all of its roots in the sector $\{z\mid 2\pi/3< \arg(z)<4\pi/3\}$ is log-concave (see~\cite[Propostion~7]{Stanley:LogConcave}).   
 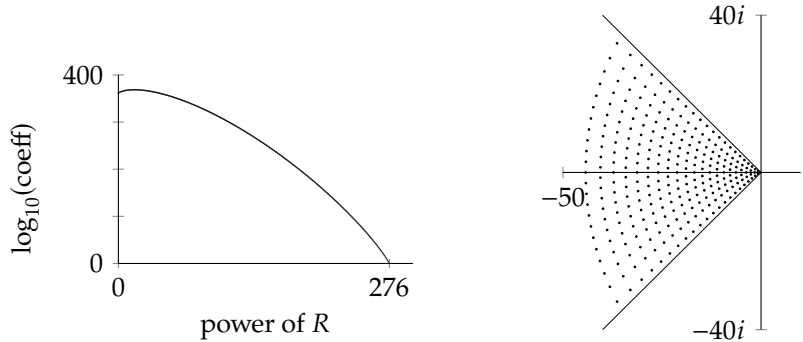
\begin{figure}[h]
\begin{center}
\begin{tikzpicture}
\begin{axis}[%tick label style={font=\footnotesize},
width = 0.45\textwidth,
%axis equal image=true,
axis x line=bottom, axis y line = left,
xmax=300,
ymax=400, 
xtick={0,276}, xticklabels={$0$ ,$276$},
ytick={0,100,...,400}, yticklabels={$0$,,,,$400$},
x axis line style={style = -},y axis line style={style = -},
xlabel={power of $R$},ylabel={$\log_{10}(\textrm{coeff})$},
yscale=0.8,
legend style={at={(1,0.1)},anchor=south east}
]

\addplot[mark size=0.1pt, only marks] file {numerator_log_coeffs.dat};
\end{axis}

\end{tikzpicture}
\qquad\qquad
\begin{tikzpicture}
\begin{axis}[%tick label style={font=\footnotesize},
width = 0.55\textwidth,
axis equal image=true,
axis x line=middle, axis y line = center,
xmin=-50, xmax=10, 
xtick={-50}, xticklabels={$-50$},
ytick={-40,40}, yticklabels={$-40i$,$40i$},
x axis line style={style = -},y axis line style={style = -},
%xlabel={power of $R$},ylabel={$\log_{10}(\textrm{coeff})$},
%yscale=0.8,
legend style={at={(1,0.1)},anchor=south east}
]
\addplot[mark size=0.3pt, only marks] file {numerator_roots.dat};
\addplot[mark=none,black,ultra thin, domain=-40:0]  {x};
\addplot[mark=none,black,ultra thin, domain=-40:0]  {-x};
\end{axis}
\end{tikzpicture}
\end{center}
\caption{On the left is a graph of the base-10 logarithm of the coefficients of the polynomial $N_{22}(R)$, which is the numerator of the magnitude of the $45$-dimensional ball; this clearly illustrates the log-concavity.  On the right is a plot in the Argand plane of the complex roots of $N_{22}(R)$; these all lie between the pictured rays $\{z\mid \arg(z)=3\pi/4\}$ and $\{z\mid \arg(z)=5\pi/4\}$.  Both sets of data were computed numerically with SageMath.}
\label{fig:LogConcavity}
\end{figure}

Two further intriguing observations are the following:
\begin{enumerate}
\item %the derivative with respect to $R$ of the magnitude 
the derivative $\frac{\dd|B_R^n|}{\dd R}$ appears to involve the Hankel determinant with a shift of one, $\det [\chi_{i+j+1}(R)]_{i,j=0}^p$;
\item %the second logarithmic derivative with respect to $R$ of the magnitude% 
the second logarithmic derivative $\frac{\dd^2\log|B_R^n|}{\dd R^2}$ appears to be minus one times a ratio of polynomials with positive coefficients,  implying, in particular, that $|B_R^n|$ is log-concave as a function of $R$ (investigating this was suggested by Mark Meckes). 
\end{enumerate}
You can verify these observations by using the SageMath notebook~\cite{Willerton:SageWorksheet}. 
It looks like there is still plenty of structure to be found in the magnitude of odd balls.

\section{Bessel potential spaces and weight distributions}
\label{section:BesselPotential}
In this section we will prove the results we need about Bessel potential spaces, proving Theorem~\ref{thm:MagnitudeWeighting} about the characterization of weight distributions for subsets of $\R^n$ and proving that the individual terms in our guess for the weight distribution of an odd ball are all in the correct Bessel potential space.

\subsection{Characterization of weight distributions in $\R^n$}
Here we will give the proof of Theorem~\ref{thm:MagnitudeWeighting} due to Meckes~\cite{Meckes:PersComm}.

We prove the required properties of a weight distribution by using properties of a corresponding potential function proved by Barcel\'o and Carbery in~\cite{BarceloCarbery}.  In~\cite[Section~5]{Meckes:MagnitudeDimensions} Meckes showed that the Bessel potential space  $H^{-i}(\R^n)$ is the topological linear dual of the following \define{Sobolev space}  of functions, equipped with a certain norm most easily defined in terms of the Fourier transform:
\[H^{i}(\mathbb{R}^{n}) := 
 \left\{f \in L^2(\mathbb{R}^{n})\bigm|  
 \tfrac{\partial^\alpha }{\partial x^\alpha}f\in L^2(\mathbb{R}^{n})\text{ for all multi-indices $\alpha$ with }|\alpha|\le i \right\}.\]
The norm  $\| {{}\cdot{}}\|_{H^i}$ this is equipped with is equivalent to, but not equal to, the more usual Hilbert space norm.
 \[ \| f\|_{\mathrm{Hilb}}:=\sqrt{\sum_{|\alpha|\le i}\left\| \tfrac{\partial^\alpha}{\partial x^\alpha} f \right\|^2_{L^2(\R^n)}}.\] 

%\begin{proof}
Meckes shows that the canonical isometry $Z\colon H^{-(p+1)}(\R^n)\to H^{p+1}(\R^n)$  between dual spaces satisfies the following.  Suppose that $w\in H^{-(p+1)}(\R^n)$ is a distribution and $h\in H^{p+1}(\R^n)$ is a function with  $h = Zw$ then 
 \[
  h(x) = \langle w,e^{-\mathrm{d}(x,\cdot)}\rangle
  \quad
  \text{and}
  \quad
  w = \frac{1}{n!\,\omega_n}(I-\Delta)^{p+1} h,
\]
where $\Delta$ is the distributional Laplacian on $\R^n$.

Now suppose that the distribution $w\in H^{-(p+1)}(\R^n)$  satisfies the hypotheses of the theorem,
then 
\begin{equation*}
    h(x) = \langle w,e^{-\mathrm{d}(x,\cdot)}\rangle = 1
    \quad \text{for all $x\in \interior{K}$}.
\end{equation*}
But since all functions in $H^{p+1}(\R^n)$ are
continuous we have 
\begin{equation}
    h(x)  = 1
    \quad \text{for all $x\in{K}$}.
        \label{eq:Meckes3}
\end{equation}

If $f$ is compactly supported in the complement of $K$, then
\begin{equation*}
    \langle(I-\Delta)^{p+1} h, f\rangle = n!\,\omega_n \langle w, f\rangle = 0,
\end{equation*}
as $w$ is supported in $K$, so 
 \begin{equation}
    (I-\Delta)^{p+1} h
= 0 \quad\text{in the weak sense on $\R^n\backslash K$.}
  \label{eq:Meckes4}
\end{equation}

By \cite[Proposition~2]{BarceloCarbery}, there is a unique $h$ in
$H^{p+1}(\R^n)$ satisfying equations~\eqref{eq:Meckes3} and~\eqref{eq:Meckes4}, so by \cite[Proposition~5.7]{Meckes:MagnitudeDimensions}, $h$ is
the potential function of $K$, and therefore $w$ is the weighting
of $K$.
%\end{proof}

The final part of the theorem, equation~\eqref{eqn:magdefn}, follows from~\cite[Proposition~4.2]{Meckes:MagnitudeDimensions}.

\subsection{Examples of distribution in $H^{-(i+1)}(\R^n)$}
\label{section:ExamplesInBesselSpace}
Here we give examples of distributions in  Bessel potential spaces.  These in fact show that our putative weight distribution $w$, defined by \eqref{eqn:defineweighting}, lives in $H^{-(p+1)}(\R^n)$ as required.

The following lemma  was suggested by Michael Renardy~\cite{Renardy:Mathoverflow} as part of the strategy to prove the theorem which follows. %Theorem~\ref{thm:NormalIntegrationInBessel}.
\begin{lemma}
Suppose $M$ is a compact subset of $\R^n$ and we have a continuous differential operator of order~$i$ of the form $\mathcal{D}=\sum_{|\alpha|\le i} e_\alpha\tfrac{\partial^\alpha}{\partial x^\alpha}$, where $\alpha$ runs over multi-indices and each $e_\alpha$ is a function on $M$.  Then the distribution $w$ defined by 
\[
  \langle w,f\rangle:=
  \int_M\mathcal{D}f \,\mathrm{d}\mathrm{vol}_{\R^n}
\]
lies in the Bessel potential space $H^{-i}(\R^n)$.
\end{lemma}
\begin{proof}
As $H^{-i}(\R^n)$ is dual to the Sobolev space $H^{i}(\R^n)$, as mentioned in the previous subsection, it suffices to show that there exists a $\Lambda\in \R$ such that for any smooth function $f\in H^i(\R^n)$ we have
\[
  \left| \langle w, f\rangle\right| \le \Lambda \| f\|_{H^i}.
\]
But as mentioned above, the norm $\| {\cdot}\|_{H^i}$ is equivalent to the Hilbert space norm $\| {\cdot}\|_{\mathrm{Hilb}}$ which is known to be equivalent to the norm $\| {\cdot}\|_{\mathrm{Alt}}$ given by 
\[
  \| f\|_{\mathrm{Alt}}:=
  \sum_{|\alpha|\le i}
  \left\| \tfrac{\partial^\alpha}{\partial x^\alpha} f \right\|_{L^2(\R^n)}.
\]  
So it suffices that to show that there exists a $\Lambda\in \R$ such that for any smooth function $f\in H^i(\R^n)$ we have
\[
  \left| \langle w, f\rangle\right| \le \Lambda \| f\|_{\mathrm{Alt}}.
\]

Let $E_\mathcal{D}:=\max_{|\alpha|\le i}\sup_{x\in M} e_\alpha(x)$; this is finite as $\mathcal{D}$ is continuous and $M$ is compact.  We have
\begin{align*}
  \left| \langle w, f\rangle\right|
  &=
  \biggl|\int_M\mathcal{D}f \,\mathrm{d}\mathrm{vol}_{\R^n} \biggr|
  =
  \biggl|\int_M\sum_{|\alpha|\le i} e_\alpha\tfrac{\partial^\alpha}{\partial x^\alpha} f \,\mathrm{d}\mathrm{vol}_{\R^n} \biggr|
  \\
  &\le
  E_\mathcal{D} \sum_{|\alpha|\le i}\int_M\left| \tfrac{\partial^\alpha}{\partial x^\alpha} f \right|\,\mathrm{d}\mathrm{vol}_{\R^n} 
  \le
  E_\mathcal{D} \sum_{|\alpha|\le i}\vol(M)^{\frac 1 2} \left\| \tfrac{\partial^\alpha}{\partial x^\alpha} f \right\|_{L^2(M)} \\
  &\le
  E_\mathcal{D} \vol(M)^{\frac 1 2} \sum_{|\alpha|\le i}\left\| \tfrac{\partial^\alpha}{\partial x^\alpha} f \right\|_{L^2(\R^n)} 
  =
  E_\mathcal{D} \vol(M)^{\frac 1 2} \| f\|_{\mathrm{Alt}}~,
\end{align*}
as required, where the second-to-last inequality comes from the Cauchy-Schwarz inequality.
\end{proof}

Note in particular that the identity operator is of order $0$ and thus of order $p+1$, this means that for the $n$-dimensional ball $B_R^n\subset\R^n$ the distribution $w$ given by $\langle w,f\rangle:=
  \int_{B_R^n} f \,\mathrm{d}\mathrm{vol}_{\R^n}$ is in $H^{-(p+1)}(\R^n)$.

\begin{thm}
\label{thm:NormalIntegrationInBessel}
Suppose $\Sigma^{n-1}\subset\R^n$ is a smooth submanifold which bounds a compact subset $M\subset\R^n$.  Define the distribution $w_i$, for $0\le i\le p$, by
\[
\langle w_i, f \rangle := \int_\Sigma \frac{\partial^i f}{\partial \nu^i}  \,\mathrm{d}\mathrm{vol}_\Sigma
\]
where $\frac{\partial}{\partial \nu} $ denotes differentiation in the outward pointing normal direction.  Then $w_i \in H^{-(p+1)}(\R^n)$.
\end{thm}
\begin{proof}  %THIS PROOF IS NOT QUITE RIGHT YET.
%As $H^{-(p+1)}(\R^n)$ is dual to the Sobolev space $H^{p+1}(\R^n)$, it suffices to prove that $\left|\langle w_i, f \rangle \right|<\infty$ whenever $f\in H^{p+1}(\R^n)$.  So, suppose $f\in H^{p+1}(\R^n)$, this means that  $D f\in L^2(\R^n)$ for every differential operator $D$ (with bounded coefficients?) of order $p+1$.  Then we will also have $D f|_M\in L^2(M)\subset L^1(M)$ as $M$ is compact.
%
Extend the normal vector field $\nu$ on $\Sigma$ to a smooth vector field $\bar \nu$ on $\R^n$, this can be done using the Tubular Neighbourhood Theorem so that $\bar \nu$ vanishes away from some neighbourhood of $\Sigma$.  We have $\frac{\partial }{\partial \nu}= \nu\cdot \nabla$, so using the Divergence Theorem we have
\begin{align*}
\langle w_i, f \rangle
&=
 \int_\Sigma \frac{\partial}{\partial \nu}\left(\frac{\partial^{i-1} f}{\partial \nu^{i-1}} \right) \,\mathrm{d}\mathrm{vol}_\Sigma
 =
  \int_\Sigma \nu\cdot \nabla \left(\frac{\partial^{i-1} f}{\partial \nu^{i-1}} \right) \,\mathrm{d}\mathrm{vol}_\Sigma
=
  \int_M \nabla^2 \left(\frac{\partial^{i-1} f}{\partial \bar\nu^{i-1}} \right) \,\mathrm{d}\mathrm{vol}_{\R^n}.
%\le
%\int_M \nabla^2 \left(\frac{\partial^{i-1} f}{\partial \bar\nu^{i-1}} \right)\,\mathrm{d}\mathrm{vol}_{\R^n}
\end{align*}
As $\nabla^2\circ\frac{\partial^{i-1} }{\partial \bar\nu^{i-1}} $ is a continuous differential operator of order~$i+1$  and hence of order~$p+1$, because $i\le p$, so by the lemma above we have that $w_i\in H^{-(p+1)}(\R^n)$.
\end{proof}

\section{Proof of the generalized key integral}
\label{section:KeyIntegral}
In this section we wish to prove Theorem~\ref{thm:GeneralKeyIntegral}, which says that for $0\le j \le p$, $\vec{s}\in \R^{2p+1}$, with $s=|\vec{s}|$ and  $R>s\ge 0
$\begin{equation}
  \int_{\vect{x}\in S^{2p}_R} \psi_j(\left | \vect{x}-\vect{s}\right|)\,\mathrm{d}\vect{x}
  =
  (-2\pi)^p  2 \e^{-R}\sum_{i=0}^{p-j} \binom{p-j}{i}\chi_{i+p}(R) \psit_{i+j}(s).
  \label{eq:GeneralizedKeyIntegral}
\end{equation}
The result we want for solving the weight equation is the case when $j=0$.

\subsection{Sketch of the proof}
Before going in to the proof, let us sketch the proof here.

The principal ingredient is an observation of Barcel\'o and Carbery~\cite{BarceloCarbery}.  We consider the functions $\psi_i(r)$ and $\psit_i(r)$ as spherically symmetric functions on $\R^n$ (where $n=2p+1$) so that $r$ represents the radial coordinate.  We denote by $\Delta$ the Laplace operator on $\R^n$.  Then the observation is that the differential operator $(I-\Delta)$ will `move us up' each of our sequences of functions, more precisely, \[(I-\Delta)\xi_i(r)=2(p-i)\xi_{i+1}(r)\] where $\xi_i(r)$ is either $\psi_i(r)$ or $\psit_i(r)$.  From this we can show that if $\xi$ is a spherically symmetric smooth function defined in a neighbourhood of the origin and $(I-\Delta)^k\xi(r)=0$ for $p\ge k\ge 1$ then $\xi(r)$ is a linear combination of the set of functions $\{\psit_i(r)\}_{i=p-k+1}^{p}$. 

By differentiating under the integral sign we deduce that 
\[
  (I-\Delta_\vec{s})^{p-j+1}\int_{\vect{x}\in S^{2p}_R} \psi_j({\left | \vect{x}-\vect{s}\right|})\,\mathrm{d}\vect{x}
  =
0
\]
(where $\Delta_\vec{s}$ indicates that we are differentiating with respect to $\vec{s}$) and thus by the above we have
\begin{equation}
  \int_{\vect{x}\in S^{2p}_R} \psi_j({\left | \vect{x}-\vect{s}\right|})\,\mathrm{d}\vect{x}
  =
  \sum_{i=0}^{p-j} a_{p,j,i}(R) \psit_{i+j}(s),
  %\tag{$\#$}
  \label{eq:IntIsLinearCombination}
  \end{equation}
for some functions $(a_{p,j,i}(R))_{i=0}^{p-j}$, which we need to identify.

In the simplest case of $j=p$ we can substitute $s=0$ and show that the sole unknown function $a_{p,p,0}(R)$ has the form we want, thus giving the integral in this case.  We work `downwards' to $j=0$ from there.

Applying the operator $(I-\Delta_\vec{s})$ to both sides of equation~\eqref{eq:IntIsLinearCombination} we find
\[
  (p-j)
  \int_{\vect{x}\in S^{2p}_R} \psi_{j+1}(\left | \vect{x}-\vect{s}\right|)\,\mathrm{d}\vect{x}
  =
  \sum_{i=0}^{p-j}  (p-j-i)a_{p,j,i}(R)\psit_{i+j+1}(s).
\]
We assume that the integral for $\psi_{j+1}$ is in the form we want, so we know the left hand side of the above equation.  From the linear independence of the set of functions $\{\psit_i\}_i$ we can deduce the correct formulae for $\{a_{p,j,i}\}_{i=0}^{p-j-1}$, which just leaves $a_{p,j,p-j}$ and this can be calculated by substituting $s=0$ into \eqref{eq:IntIsLinearCombination}.  In this way we obtain the generalized key integral~\eqref{eq:GeneralizedKeyIntegral}, and, in particular for $j=0$, the key integral, Theorem~\ref{thm:KeyIntegral}.

\subsection{Basic properties of our functions}
Here we will obtain some useful facts about the sequences of functions $(\psi_i(r))_{i=0}^\infty$ and $(\psit_i(r))_{i=0}^\infty$.  Firstly their initial terms are as follows:
\[
  \psi_0(r)= \e^{-r}; \qquad\psit_0(r)= \cosh(r)
  \]
 Both sequences satisfy the same recurrence relation
\begin{equation}
  \xi_{i+1}(r)=-\tfrac{1}{r}\xi_{i}'(r).
  \label{eq:DiffRecursion}
\end{equation}
For $r\ne 0$, the second sequence is the even part of the first: $\psit_i(r)=\frac12(\psi_i(r)+\psi_i(-r))$.  It is not completely obvious, but is easy to prove, that $\psit(0)$ is well-defined.
\begin{prop}
\label{prop:psiTilde0}
For $i\ge 0$ the function $\psit_i(r)$ is an even function defined on all of $\R$, in particular at $r=0$.  It has Taylor expansion
\[
  \psit_i(r)=(-1)^i \sum_{k=0}^\infty \frac{r^{2k}}{(2k)!\,\cdot(2k+1)\cdot(2k+3)\cdots(2k+2i-1)},
\]
and so
\[
  \psit_i(0)= \frac{(-1)^i}{1\cdot 3\cdot 5\cdots (2i-1)} .
\]
\end{prop}
\begin{proof}
It suffices to prove the expression for the Taylor series.  This is a straightforward induction using the recursive definition $\psit_{i+1}(r)=-\tfrac{1}{r}\psit_{i}'(r)$.
\end{proof}
The functions satisfy the following recursion relation.%, this usually expresses $\psi_{i+1}$ in terms of $\psi_i$ and $\psi_{i-1}$, but we will see the reason for writing it in this form below.
\begin{thm} For $i\ge 1$, both sequences $(\psi_i(r))_{i=0}^\infty$ and $(\psit_i(r))_{i=0}^\infty$ satisfy
\[\xi_{i+1}(r)=(\xi_{i-1}(r)+(2i-1)\xi_i(r))/r^2.\]
\end{thm}
\begin{proof}
We will prove the formula inductively.  It is easily checked for the case $i=1$ for both sequences of functions.  Then, by the recursion formula~\eqref{eq:DiffRecursion},
\begin{align*}
  \xi_{i+1}(r)
  &=-\tfrac{1}{r}\xi_i'(r)
  \\
  &=
  -\frac{1}{r}\left(\frac{\xi_{i-2}(r)+(2i-3)\xi_{i-1}(r)}{r^2}\right)'
  \\
  &=
  -\frac{\xi_{i-2}'(r)r^2-\xi_{i-2}(r)2r+(2i-3)\xi_{i-1}'(r)r^2-(2i-3)\xi_{i-1}(r)2r}{r^5}
 \\
 &=
  -\frac{-\xi_{i-1}(r)r^2-(2i-3)\xi_{i}(r)r^2-2\left(\xi_{i-2}(r)+(2i-3)\xi_{i-1}(r)\right)}{r^4}
  \\
  &=
  \frac{\xi_{i-1}(r)r^2+(2i-3)\xi_{i}(r)r^2+2r^2\xi_{i}(r)}{r^4}
\\
  &=
  \frac{\xi_{i-1}(r)+(2i-1)\xi_{i}(r)}{r^2},
\end{align*}
and the result follows by induction.
\end{proof}
From this we can express the relation between the function $\psi_i$ and the reverse Bessel polynomial $\chi_i$.
\begin{prop} For $i\ge 0$
\label{prop:chiAndpsi}
\[\chi_i(r) = \e^r r^{2i}\psi_i(r).\]
\end{prop}
\begin{proof}
By the above theorem we have 
\[\psi_{i+2}(r)=(\psi_{i}(r)+(2i+1)\psi_{i+1}(r))/r^2.\]
Multiplying through by $\e^rr^{2i+4}$ we find that $\e^r r^{2i}\psi_i(r)$ satisfies the defining recursion relation~\eqref{eq:ChiRecursion} for the reverse Bessel polynomials.  You can check that for $i=0,1$ the functions agree, therefore they agree for all $i$.
\end{proof}

The above recursive relation can be written in the following fashion.
\begin{equation}
  \xi_{i-1}(r) = -(2i-1)\xi_i(r) + r^2\xi_{i+1}(r).
  \label{eq:AltRecursion}
\end{equation}
If this is thought of as writing the `decrease the index by one' operation as `identity times a scalar depending on the index plus increase index by one times $r^2$', then decreasing the index by $k$ is doing the sum of those two operations $k$ times.  In general, if we have two \emph{commuting} operations $A$ and $B$ then there is a Leibniz formula $(A+B)^k=\sum_{m=0}^k\binom{k}{m}A^{k-m} B^m$.  The following expression, which we will use in Theorem~\ref{thm:GeneralKeyIntegral} below, ought to be seen in that context, but I don't see how to make that rigourous.
\begin{lemma}
\label{lemma:Leibniz}
For $i\ge 1$ and $0\le k \le i$, both sequences $(\psi_i(r))_{i=0}^\infty$ and $(\psit_i(r))_{i=0}^\infty$ satisfy the following formula:
\[
  \xi_{i-k}(r) = \sum_{m=0}^k\binom{k}{m}
     \bigg(\prod_{\ell=i+1-(k-m)}^i -(2\ell-1)\bigg)r^{2m}\xi_{i+m}(r).
\]
\end{lemma}
\begin{proof}
Let's simplify the notation by writing $A^j=\prod_{\ell=i+1-j}^i -(2\ell-1)$.  The formula holds vacuously for $k=0$.  We use the recursion relation~\eqref{eq:AltRecursion} and then proceed inductively.
\begin{align*}
  \xi_{i-k}
  &=
  -(2(i-k)+1)\xi_{i-k+1}(r)+r^2\xi_{i-k+2}(r)
  \\
  &=
  -(2(i-k)+1)\sum_{m=0}^{k-1}\binom{k-1}{m}
     A^{k-1-m}r^{2m}\xi_{i+m}(r)
   \\
   &\quad
   + r^2\sum_{m=0}^{k-2}\binom{k-2}{m}
     A^{k-2-m}r^{2m}\xi_{i+m}(r)
  \\
  &=
  \sum_{m=0}^{k-1}\binom{k-1}{m}(-(2(i-k)+1))
     A^{k-1-m}r^{2m}\xi_{i+m}(r)
   \\
   &\quad
   + \sum_{m=0}^{k-2}\binom{k-2}{m}
     A^{k-2-m}r^{2m+2}\bigl(-(2i+2m+1)\xi_{i+m+1}(r)
     +r^2\xi_{p+m+2}\bigr)
%   \\
%   &\qquad\qquad\qquad\qquad\qquad\qquad+r^2\xi_{p+m+2}\bigr)
   \\
   &=
  \sum_{m=0}^{k-1}\binom{k-1}{m}(-(2(i-k)+1))
     A^{k-1-m}r^{2m}\xi_{i+m}(r)
   \\
   &\quad
   + \sum_{m=0}^{k-2}\binom{k-2}{m}
     A^{k-2-m}r^{2m+2}\bigl(-(2i+2m+1)\bigr)\xi_{i+m+1}(r)
   \\
   &\quad
   + \sum_{m=0}^{k-2}\binom{k-2}{m}
     A^{k-2-m}r^{2m+4}\xi_{p+m+2},
  \\
  \intertext{reindexing the sums,}
   &=
  \sum_{m=0}^{k}
   \biggl[ -\binom{k-1}{m}(2i-2k+1) -\binom{k-2}{m-1}(2i+2m-1)
   \\
   &\quad\qquad
   -\binom{k-2}{m-2}(2i+2m-2k+1) \biggr]  A^{k-1-m}r^{2m}\xi_{i+m}(r),
  \\
 \intertext{using $\binom{k-1}{m}m=\binom{k-2}{m-1}(k-1)$,}
  &=
  \sum_{m=0}^{k}
   \biggl[ \binom{k-1}{m}+\binom{k-2}{m-1}
   +\binom{k-2}{m-2}\biggr]
   \\
   &\quad\qquad \bigl(-(2i-2(k-m)+1)\bigr) A^{k-m-1}r^{2m}\xi_{i+m}(r)
   \\
  &=
  \sum_{m=0}^{k}
   \binom{k}{m}A^{k-m}r^{2m}\xi_{i+m}(r),
\end{align*}
as required.
\end{proof}
We finish this subsection by showing that our functions satisfy a second order differential equation.
\begin{thm}\label{thm:DiffEqn}
For $i\ge 0$ and $r>0$, both $\psi_i$ and $\psit_i$ satisfy the following differential equation:
%\begin{enumerate}
\[
\xi_i''(r)+\tfrac{2i}{r}\xi_i'(r)- \xi_i(r)=0.
%\quad\text{and}\quad
%\psit_i''(r)+\tfrac{2i}{r}\psit_i'(r)- \psit_i(r)=0.
\]
%\item $r^2 \psi_{i+2}= \psi_i+(1+2i)\psi_{i+1}$.
%\end{enumerate}
\end{thm}
\begin{proof}
%\begin{enumerate}
This is proved by induction.   It is clearly true for both $\psi_0$ and $\psit_0$.  To prove the inductive step, begin by substituting in the inductive definition~\eqref{eq:DiffRecursion}.
\begin{align*}
\xi_{i+1}''+\tfrac{2i+2}{r}\xi_{i+1}' &=
(-\tfrac{1}{r}\xi_{i}')''+\tfrac{2i+2}{r}(-\tfrac{1}{r}\xi_{i}')'\\
&= -(\tfrac{2}{r^3} \xi_{i}' -\tfrac{2}{r^2} \xi_i'' +\tfrac{1}{r} \xi_i''')  -\tfrac{2i+2}{r}(-\tfrac{1}{r^2}\xi_i' + \tfrac{1}{r} \xi_i'') \\
&= -\tfrac{1}{r}[\xi_i''' +\tfrac{2i}{r}\xi_i'' -\tfrac{2i}{r^2}\xi_i']\\
&= -\tfrac{1}{r} [\xi_i''+\tfrac{2i}{r}\xi_i' ]' =-\tfrac{1}{r}\xi_i' =\xi_{i+1}.
\end{align*}
%\begin{align*}
%\psi_{i+1}''+\tfrac{2i+2}{r}\psi_{i+1}' &- \psi_{i+1} \\
%&=(-\tfrac{1}{r}\psi_{i}')''+\tfrac{2i+2}{r}(-\tfrac{1}{r}\psi_{i}')'-(-\tfrac{1}{r}\psi_{i}')\\
%&= -[(\tfrac{2}{r^3} \psi_{i}' -\tfrac{2}{r^2} \psi_i'' +\tfrac{1}{r} \psi_i''')  +\tfrac{2i+2}{r}(-\tfrac{1}{r^2}\psi_i' + \tfrac{1}{r} \psi_i'') - \tfrac{1}{r}\psi_{i}']\\
%&= -\tfrac{1}{r}[\psi_i''' +\tfrac{2i}{r}\psi_i'' -\tfrac{2i}{r^2}\psi_i'+\psi_i']\\
%&= -\tfrac{1}{r} [\psi_i''+\tfrac{2i}{r}\psi_i' -\psi_i]' =0.
%\end{align*}
The next to last equality comes from the inductive hypothesis.
%\item We begin by using the inductive definition in the left hand side.
%\begin{align*}
%r^2\psi_{i+2}&= r^2(-\tfrac{1}{r}(-\tfrac{1}{r}\psi_i')')\\
%&= r^2(-\tfrac{1}{r}(-\tfrac{1}{r} \psi_i'' -\tfrac{1}{r^2}\psi_i'))\\
%&= \psi_i'' -\tfrac{1}{r} \psi_i' = \psi_i -\tfrac{2i}{r}\psi_i' -\tfrac{1}{r} \psi_i' \\
%&= \psi_i + (1+2i)\psi_{i+1},
%\end{align*}
%where the differential equation from above was used to substitute in for $\psi_i''$.
%\end{enumerate}
\end{proof}

\subsection{The Laplacian and spherically symmetric functions on $\R^n$}
As throughout this paper $n=2p+1$ is a fixed odd integer.  Here we will think of functions of $r$ as being spherically symmetric functions on $\R^n$.  Our proof of the generalized key integral will rely on the behaviour on the set of functions $\{\psi_i\}\cup\{\psit_i\}$ with respect to a certain differential operator defined in terms of the Laplacian operator $\Delta$.   Indeed, we will show that this set of functions spans the solution set of a certain differential equation.
Theorem~\ref{thm:LadderPsi} and Corollary~\ref{cor:Annihilate} were obtained in~\cite{BarceloCarbery}, but the proofs given here are more direct.

The Laplacian operator $\Delta$ is a differential operator defined on functions on $\R^n$; for a spherically symmetric function $g(r)$ the Laplacian is given by
  \[
  \Delta g(r) = g''(r) + \tfrac{n-1}{r} g'(r).
  \]
A fundamental property (observed in~\cite{BarceloCarbery}) is that the operator $(I-\Delta)$ moves us along the sequence of functions.
%
%We can now get to the main property of the sequence $(\psi_i)$ that we are interested in.  We will switch perspectives, fix a positive integer $n$ and consider each $\psi_i$ as a spherically symmetric function on $\R^n\backslash\{0\}$.  Really we should denote $\Psi_i\colon\R^n\backslash\{0\}\to \R$ with $\Psi_i(x)\coloneqq \psi_i(|x|)$, but we will abuse notation and just use $\psi_i$ in the two senses, with $r$ being interpreted at the radial coordinate.
%
%There is the Laplacian operator $\Delta$ on functions on $\R^n$; on a spherically symmetric function  $\psi(r)$ the Laplacian is given by \[\Delta \psi(r) = \psi''(r) + \tfrac{n-1}{r} \psi'(r).\]
%We can now reveal the main property of interest.
\begin{thm}
\label{thm:LadderPsi}For $i\ge 0$ both sequences $(\psi_i(r))_{i=0}^\infty$ and $(\psit_i(r))_{i=0}^\infty$ satisfy the following formula:
\[
(I- \Delta) \xi_i(r) = 2(p-i)\xi_{i+1}(r).
\]
\end{thm}
\begin{proof}
The result follows from the differential equation in Theorem~\ref{thm:DiffEqn} and the defining recurrence relation~\eqref{eq:DiffRecursion}:
\begin{align*}(I- \Delta) \xi_i(r) &= \xi_i(r) - \xi_i''(r) - \tfrac{n-1}{r} \xi_i'(r)\\
&= \tfrac{2i}{r}\xi'_i(r) - \tfrac{n-1}{r} \xi_i'(r)\\
&= (n-1-2i)\xi_{i+1}(r).
\qedhere
\end{align*}
\end{proof}

\begin{cor}
\label{cor:Annihilate}
For $k\ge 1$ and $i\ge 0$ both sequences $(\psi_i(r))_{i=0}^\infty$ and $(\psit_i(r))_{i=0}^\infty$ satisfy the following formula
\begin{align*}
(I- \Delta)^k \xi_i &= 2^k(p-i)(p-i-1)\dots(p-i-k+1)\xi_{i+k}.
\end{align*}
So if, further, $i\le p$ and $i+k>p$ then
\[ (I- \Delta)^k \xi_i =0 .\]
\end{cor}
%Of course, if $n$ is odd, then $(I-\Delta)\psi_{(n-1)/2} =0$, so we find that $I-\Delta$ is nilpotent on certain $\psi_i$.
%\begin{cor}
%If $n$ is odd and $0\le i\le (n-1)/2$ then
%\[(I-\Delta)^{(n+1)/2} \psi_i=0,\]
%in particular, taking $i=0$,
%\[(I-\Delta)^{(n+1)/2} e^{-r}=0.\]
%\end{cor}
We can use this to show how subsets of our sets of functions span certain solution sets.
\begin{thm}
For $p\ge k \ge 1$, if $g$ is a spherically symmetric smooth function defined on a neighbourhood of the origin in $\R^n$ with $(I-\Delta)^kg(r)=0$ then there exists a set of constants $\{c_i\}_{i=p-k+1}^p$ such that
\[g(r)=\sum_{i=p-k+1}^p c_i \psit_i(r).\]
\end{thm}
\begin{proof}
By Corollary~\ref{cor:Annihilate}, the set $\{\psi_{p-k+1},\dots,\psi_{p},\psit_{p-k+1},\dots,\psit_{p}\}$ gives us $2k$ linearly independent solutions to the order $2k$ linear differential equation $(I-\Delta)^k g(r)=0$, so this set spans the space of solutions.  The function $\psi_i$, for $i\ge 1$ has a singularity of order $2i-1$ and the function $\psi_0(r)$ is $\e^{-r}$ which is not differentiable at the origin as a spherically symmetric function; thus any solution which is smooth at the origin must be a linear combination of the set of functions $\{\psit_{p-k+1},\dots,\psit_{p}\}$.
%If we define $\overline\psi_0(r):= e^r$ and $\overline\psi_{i+1}(r):=-\tfrac{1}{r}\overline\psi_i(r)$ then you can easily see $\overline\psi_i\in e^{r}\N[\tfrac{1}{r}]$.  All of the arguments used above for the sequence $(\psi_i)$ go through unchanged for $(\overline \psi_i)$, so in particular we find that
%$(I-\Delta)^{(n+1)/2} \psi_i=0$ for $0\le i\le (n-1)/2$.  This means that the set $\{\psi_0,\dots,\psi_{(n-1)/2},\psi_0,\dots,\psi_{(n-1)/2}\}$ gives us $n-1$ linearly independent solutions to $(I-\Delta)^{(n+1)/2} g=0$ which is an order $n-1$ linear ordinary differential equation, so our solutions span the space of solutions.  However, for a solution to decay, as required, it must be a linear combination of the first half of those.
\end{proof}

\subsection{Proof of the theorem}
We will begin with the following first approximation to Theorem~\ref{thm:GeneralKeyIntegral} and then find the functions $a_{p,j,i}(R)$.
\begin{prop}
\label{prop:IntIsLinearComb}
For $p\in \N_{>0}$ and  $0\le j \le p$,  there is a set of  functions $\{a_{p,j,i}(R)\}_{i=0}^{p-j}$ such that for $0<R$ and $0\le s < R$
\[
  \int_{\vect{x}\in S^{2p}_R} \psi_j(\left | \vect{x}-\vect{s}\right|)\,\mathrm{d}\vect{x}
  =
  2(-2\pi)^p \sum_{i=0}^{p-j} \binom{p-j}{i}a_{p,j,i}(R) \psit_{i+j}(s).
  \]
\end{prop}
\begin{proof}
For the moment, fix $R>0$.  Observe that $\psi_j(r)$ is smooth for $r\ne 0$ so for fixed $\vect{x}\in \R^{2p+1}$ with $|\vect{x}|=R$ and $|\vect{s}| < R$ we have $\psi_j(|\vect{x}-\vect{s}|)$ is a smooth function of $\vect{s}$.   As $(I-\Delta_{\vect{s}})^{p-j+1}\psi_j(|\vect{s}|)=0$, and the Laplacian is translation invariant, $(I-\Delta_{\vect{s}})^{p-j+1}\psi_j(|\vect{x}-\vect{s}|)=0$.  Averaging over a sphere will not reduce smoothness so for all $|\vect{s}|<R $ we have $  \int_{\vect{x}\in S^{2p}_R} \psi_j(\left | \vect{x}-\vect{s}\right|)\,\mathrm{d}\vect{x}
$ is a smooth, spherically symmetric function of $\vect{s}$.  Then differentiating under the integral sign we find
  \begin{align*}
   (I-\Delta_{\vect{s}})^{p-j+1} \int_{\vect{x}\in S^{2p}_R} \psi_j(\left | \vect{x}-\vect{s}\right|)\,\mathrm{d}\vect{x}
  &= \int_{\vect{x}\in S^{2p}_R} (I-\Delta_{\vect{s}})^{p-j+1}\psi_j(\left | \vect{x}-\vect{s}\right|)\,\mathrm{d}\vect{x}
  =0.
\end{align*}
Now we can just apply the lemma above to find that there is a set of constants $\{c_i\}_{i=0}^{p-j}$ so that
\[\int_{\vect{x}\in S^{2p}_R} \psi_j(\left | \vect{x}-\vect{s}\right|)\,\mathrm{d}\vect{x}=
 \sum_{i=0}^{p-j} c_{i} \psit_{i+j}(s).\]
But these `constants' depend on our fixed $R$, so rescaling these functions of $R$ appropriately we find  a set of  functions $\{a_{p,j,i}(R)\}_{i=0}^{p-j}$ such that
\[
  \int_{\vect{x}\in S^{2p}_R} \psi_j(\left | \vect{x}-\vect{s}\right|)\,\mathrm{d}\vect{x}
  =
  2(-2\pi)^p \sum_{i=0}^{p-j} \binom{p-j}{i}a_{p,j,i}(R) \psit_{i+j}(s),
  \]
for $s<R$, as required.
\end{proof}

We now just have to prove that for given $p$ and $0\le j\le p$ we have $a_{p,j,i}(R)= R^{2p+2i}\psi_{p+i}(R)$.  I cannot see any obvious reason why these are independent of $j$.  We will inductively work down from $j=p$ to $j=0$.  First we will do the base case of $j=p$.  (As mentioned in the introduction, this is equivalent to a known result in the literature about modified spherical Bessel functions.)
\begin{prop}
\label{prop:jEqualspVersion}
For $R> s =|\vec{s}|$ we have
\[
  \int_{\vect{x}\in S^{2p}_R} \psi_p(\left | \vect{x}-\vect{s}\right|)\,\mathrm{d}\vect{x}
  =
  2(-2\pi)^p R^{2p}\psi_p(R) \psit_{p}(s).
  \]
\end{prop}
%Details will appear elsewhere.
\begin{proof}
By Proposition~\ref{prop:IntIsLinearComb} we have
\begin{equation}
\label{eq:sEquals0}
  \int_{\vect{x}\in S^{2p}_R} \psi_p(\left | \vect{x}-\vect{s}\right|)\,\mathrm{d}\vect{x}
  =
  2(-2\pi)^p a_{p,p,0}(R) \psit_{p}(s),
  \end{equation}
for some function $a_{p,p,0}(R)$.   Substituting $\vect{s}=\vect{0}$ and writing $\sigma_{2p}$ for the volume of the unit $2p$-dimensional sphere, the left hand side of~\eqref{eq:sEquals0} becomes
\begin{align*}
  \int_{\vect{x}\in S^{2p}_R} \psi_p(\left | \vect{x}\right|)\,\mathrm{d}\vect{x}
  &=
  \int_{\vect{x}\in S^{2p}_R} \psi_p(R)\,\mathrm{d}\vect{x}
  =
  \sigma_{2p}R^{2p}\psi_p(R)
  \\
  &=
  \frac{2(2\pi)^p}{1\cdot 3\cdot 5\cdots (2p-1)}R^{2p}\psi_p(R)
  \\
  &=
  2(-2\pi)^p\psit_p(0)R^{2p}\psi_p(R),
\end{align*}
where the formula for $\psit_p(0)$ was given in Theorem~\ref{prop:psiTilde0}.  The result follows from comparing with the right hand side of equation~\eqref{eq:sEquals0}.
  \end{proof}

Now we can prove the general case.
%\begin{thm}
%%\label{thm:GeneralKeyIntegral}
%For $p\in \N_{>0}$, $0\le j \le p$, and for $0<R$ and $0\le s < R$ we have
%\[
%  \int_{\vect{x}\in S^{2p}_R} \psi_j(\left | \vect{x}-\vect{s}\right|)\,\mathrm{d}\vect{x}
%  =
%  2( -2\pi)^p\sum_{i=0}^{p-j} \binom{p-j}{i}R^{2p+2i}\psi_{p+i}(R) \psit_{j+i}(s).
%  \]
%\end{thm}
\begin{proof}[Proof of Theorem~\ref{thm:GeneralKeyIntegral}]
We work inductively downwards.  Assume that for given $p$, the theorem holds for all $j$ satisfying $0\le k<j\le p$.  We will prove that it holds for $j=k$.  We know that
\begin{align*}
  \int_{\vect{x}\in S^{2p}_R} \psi_k(\left | \vect{x}-\vect{s}\right|)\,\mathrm{d}\vect{x}
  &=
  2(-2\pi)^p \sum_{i=0}^{p-k} \binom{p-k}{i}a_{p,k,i}(R) \psit_{i+k}(s).
  \\
\intertext{Thus}
  (I-\Delta_{\vect{s}})\int_{\vect{x}\in S^{2p}_R} \psi_k(\left | \vect{x}-\vect{s}\right|)\,\mathrm{d}\vect{x}
  &=
  (I-\Delta_{\vect{s}})2(-2\pi)^p \sum_{i=0}^{p-k} \textstyle\binom{p-k}{i}a_{p,k,i}(R) \psit_{i+k}(s).
  \\
\intertext{So}
  \int_{\vect{x}\in S^{2p}_R} (I-\Delta_{\vect{s}})\psi_k(\left | \vect{x}-\vect{s}\right|)\,\mathrm{d}\vect{x}
  &=
  2(-2\pi)^p \sum_{i=0}^{p-k} \textstyle\binom{p-k}{i}a_{p,k,i}(R) (I-\Delta_{\vect{s}})\psit_{i+k}(s).
  \\
\intertext{Then, by Theorem~\ref{thm:LadderPsi} and the translation invariance of $\Delta_{\vect{s}}$,}
  \int_{\vect{x}\in S^{2p}_R} (p-k)\psi_{k+1}(\left | \vect{x}-\vect{s}\right|)\,\mathrm{d}\vect{x}
  &=
  2(-2\pi)^p \sum_{i=0}^{p-k} \textstyle\binom{p-k}{i}a_{p,k,i}(R) (p-k-i)\psit_{i+k+1}(s).
\end{align*}
By the inductive hypothesis, the left hand side of the above equation has the form
\[  2( -2\pi)^p\sum_{i=0}^{p-(k+1)} \textstyle\binom{p-(k+1)}{i}(p-k)R^{2p+2i}\psi_{p+i}(R) \psit_{k+1+i}(s).\]
But as $ \binom{p-k-1}{i}(p-k) = \binom{p-k}{i}(p-k-i)$, the previous equation becomes
%\begin{multline*}
\[
  \sum_{i=0}^{p-k-1} {\textstyle\binom{p-k-1}{i}}(p-k)R^{2p+2i}\psi_{p+i}(R) \psit_{k+1+i}(s)
  =  \sum_{i=0}^{p-k-1} \textstyle\binom{p-k-1}{i}(p-k)a_{p,k,i}(R) \psit_{k+1+i}(s).
\]
%\end{multline*}
From the linear independence of the set of functions $\{\psit_i(s)\}_{i=0}^p$ we deduce that $a_{p,k,i}(R) =R^{2p+2i}\psi_{p+i}(R)$ for $i=0,\dots,p-k-1$.  It just remains to find $a_{p,k,p-k}$.  To do this we will consider the case $s=0$.

We now know
%\begin{multline}
\begin{equation}
 \frac{1}{2( -2\pi)^p} \int_{\vect{x}\in S^{2p}_R} \psi_k(\left | \vect{x}-\vect{s}\right|)\,\mathrm{d}\vect{x}
  =
  \sum_{i=0}^{p-k-1} \binom{p-k}{i}R^{2p+2i}\psi_{p+i}(R) \psit_{k+i}(s)
  + a_{p,k,p-k}(R)\psit_{p}(s) .
  \label{eq:sIsZero}
\end{equation}
%  \end{multline}
Taking $\vect{s}=\vect{0}$, and writing $\sigma_{2p}$ for the volume of the unit $2p$-dimensional sphere, the left hand side becomes
\[ \frac{1}{2( -2\pi)^p} \int_{\vect{x}\in S^{2p}_R} \psi_k(\left | \vect{x}\right|)\,\mathrm{d}\vect{x}
= \frac{\sigma_{2p}R^{2p}\psi_k(R)}{2( -2\pi)^p}  = \psit_p(0)R^{2p}\psi_k(R)
\]
because as noted above, in Proposition~\ref{prop:jEqualspVersion}, $\sigma_{2p}=2(-2\pi)^p\psit_p(0)$.

Writing $\psi_k(R)=\psi_{p-(p-k)}(R)$ we can use Lemma~\ref{lemma:Leibniz} and find this is equal to
\begin{align*}
\psit_p(0)R^{2p}
  \sum_{i=0}^{p-k}\binom{p-k}{i}
     \bigg(\prod_{\ell=1+k+i}^p -(2\ell-1)\bigg)R^{2i}\psi_{p+i}(R)
     &=\psit_p(0)R^{2p}
  \sum_{i=0}^{p-k}{\binom{p-k}{i}}
     \frac{\psit_{k+i}(0)}{\psit_{p}(0)}R^{2i}\psi_{p+i}(R)
  \\
  &=
  \sum_{i=0}^{p-k}{\binom{p-k}{i}}
    \psit_{k+i}(0)R^{2p+2i}\psi_{p+i}(R).
\end{align*}
Comparing this with the right hand side of equation~\eqref{eq:sIsZero} shows that $a_{p,k,p-k}(R)$ is of the required form and the theorem is proved.
\end{proof}
We have now proved the key integral identity for solving the weight equation.

%\subsection{A plausible strategy for a proof}\label{section:PlausibleStrategy}
%There are integral expressions as in the following three propositions.  One might hope to use them to prove the key result.
%\begin{prop}
%For $i\ge 1$ and $R>0$ we have the following two integral forms:
%\begin{align*}
%  \chi_i(R)
%  &=
%  \frac{e^R R^{2i}}{2^{i-1} (i-1)!}\int_{t=0}^\infty e^{-R\cosh t}\sinh^{2i-1}t\, \mathrm d t\\
%  &=
%  \frac{e^R R}{2^{i-1} (i-1)!}\int_{y=R}^\infty
%  e^{-y}(y^2-R^2)^{i-1}\, \mathrm d y.
%\end{align*}
%\end{prop}
%
%\begin{prop}
%For $i\ge 1$ and $s> 0$ we have the following two integral forms:
%\begin{align*}
%  \psit_i(s)
%  &=
%  \frac{(-1)^{i}}{2^{i} (i-1)!}\int_{\theta=0}^\pi e^{s\cos \theta}\sin^{2i-1}\theta\, \mathrm d \theta\\
%  &=
%  \frac{(-1)^{i}}{2^{i} (i-1)!s^{2i-1}}\int_{x=-s}^s
%  e^{x}(s^2-x^2)^{i-1}\, \mathrm d x.
%\end{align*}
%\end{prop}
%
%\begin{prop} With notation as in the above conjecture, so $R>s>0$, we have the following two forms for the integral over the sphere:
%\begin{align*}  &\frac{1}{n!\,\omega_n}
%    \int_{\vec{x}\in\sphere}
%    e^{-\left|\vec{x}-\vec{s}\right|}\,\mathrm{d}\vec{x}\\
%&= \frac{ R^{2p}}{p!(p-1)!(2)^{2p}}\int_{\theta=0}^{\pi}e^{-\sqrt{R^2+s^2-2Rs\cos \theta}}\sin^{2p-1}\theta\,\mathrm{d}\theta\\
% &=\frac{ R}{p!(p-1)!(4s)^{2p-1}}\int_{\rho=R-s}^{R+s}e^{-\rho}((R+s)^2-\rho^2)^{p-1}
%                                        (\rho^2-(R-s)^2)^{p-1}\rho\,\mathrm{d}\rho.
%  \end{align*}
%\end{prop}
%
%\subsection{A plausibility argument}
%
\section{Proofs of the corollaries of the key integral}
\label{section:ProofsOfCorollaries}
In this section we will prove Corollaries~\ref{cor:FirstCor} and~\ref{cor:SecondCor} which are needed for solving the weight equation.
\subsection{Integration over a ball}
\label{section:ProofOfFirstCor}
Here we will prove Corollary~\ref{cor:FirstCor} about integration of $\e^{-|\vect{x}-\vect{s}|}$ over the ball of given radius by using the result on the integration over the spheres of various radius.  

Note first the following two lemmas.
\begin{lemma} For $n\ge 1$ and $\vect{s}\in \R^n$ we have
\[\frac{1}{n!\,\omega_n}\int_{\vect{x}\in \R^n} \e^{-|\vect{x}-\vect{s}|}\mathrm{d}\vect{x} = 1.\]
\end{lemma}
\begin{proof}
By translating the origin and noting that the volumes of the unit $(n-1)$-sphere and unit $n$-ball  are related by $\sigma_{n-1}  = n\omega_n$, we have
\begin{align*}
\int_{\vect{x}\in \R^n} \e^{-|\vect{x}-\vect{s}|}\mathrm{d}\vect{x}
  &=
  \int_{\vect{x}\in \R^n} \e^{-|\vect{x}|}\mathrm{d}\vect{x}
  =
  \int_{r=0}^{\infty} \e^{-r} r^{n-1}\sigma_{n-1}\mathrm{d} r
  \\
  &=
  \Gamma(n) \sigma_{n-1}
  =
  (n-1)!\, n\omega_n
  =n!\,\omega_n.
\end{align*}
\end{proof}

\begin{lemma}
For $i=0,1,2,\dots$ and $R>0$ we have
  \[\int_R^\infty \chi_i(r) e^{-r} \,\mathrm{d}r = \frac{e^{-R}\chi_{i+1}(R)}{R}.\]
\end{lemma}
\begin{proof}
Observe first that
\begin{align*}
  \frac{\dd}{\dd r}\left(-\frac{e^{-r}\chi_{i+1}(r)}{r}\right)
  &=
  \frac{\dd}{\dd r}\left(-\psi_{i+1}(r)r^{2i+1}\right)
  %\\
 % &
  =
  r\psi_{i+2}(r)r^{2i+1}-\psi_{i+1}(r)(2i+1)r^{2i}
 \\
 &
  =
  \e^{-r}r^{-2}\left(\chi_{i+2}(r)-(2i+1)\chi_{i+1}(r)\right)
  %\\
 % &
  =
  e^{-r}\chi_i(r).
\end{align*}
So, by the Fundamental Theorem of Calculus,
\[
  \int_{r=R}^\infty e^{-r} \chi_i(r)\,\dd r
  =
  \left[-\frac{e^{-r}\chi_{i+1}(r)}{r}\right]^\infty_{r=R}
  =
  \frac{e^{-R}\chi_{i+1}(R)}{R},
  \]
as required.
\end{proof}
We can now do the integral over the ball:
\begin{align*}
\frac{1}{n!\,\omega_n}\int_{\vect{x}\in \ball} \e^{-|\vect{x}-\vect{s}|}\mathrm{d}\vect{x}
  &=
  \frac{1}{n!\,\omega_n}\left(\int_{\vect{x}\in \R^n} \e^{-|\vect{x}-\vect{s}|}\mathrm{d}\vect{x}
  - \int_{|\vect{x}|>R} \e^{-|\vect{x}-\vect{s}|}\mathrm{d}\vect{x} \right)\\
  &=
  1-\int_{r=R}^\infty \frac{1}{n!\,\omega_n}\int_{\vect{x}\in S^{n-1}_r} \e^{-|\vect{x}-\vect{s}|}\mathrm{d}\vect{x}\,\mathrm{d}r
  \\
  &=
  1-\int_{r=R}^\infty \frac{(-1)^p\e^{-r}}{2^p p!} \sum_{i=0}^p %\textstyle
     \binom{p}{i}\chi_{i+p}(r)\psit_{i}(s)\,\mathrm{d}r
  \\
  &=
  1-\frac{(-1)^p}{2^p p!} \sum_{i=0}^p %\textstyle
     \binom{p}{i}\left(\int_{r=R}^\infty \e^{-r}\chi_{i+p}(r)\,\mathrm{d}r\right) \psit_{i}(s)
  \\
  &=
  1-\frac{(-1)^p \e^{-R}}{2^p p!} \sum_{i=0}^p %\textstyle
     \binom{p}{i}\frac{\chi_{i+p+1}(R)}{R} \psit_{i}(s),
\end{align*}
which is as required.

\subsection{Integration of the normal derivatives over a sphere}
\label{section:ProofOfSecondCor}
Here we will prove Corollary~\ref{cor:SecondCor} about integration of the normal derivatives of $\e^{-|\vect{x}-\vect{s}|}$, for fixed $\vect{s}$, over a sphere of given radius.

We want to do differentiation under the integral sign, but the region we are integrating over --- the radius $R$ sphere --- depends on the thing, $R$, we are differentiating with respect to.  The trick is to rescale and write the integral as an integral over the unit radius sphere.  Then differentiation in the normal direction just becomes differentiation with respect to $R$.  We will write $\hat{\vect{ x}}$ for a vector on the \emph{unit} sphere $S^{2p}$.
\begin{align*}
\frac{1}{n!\,\omega_n}\int_{\vect{x}\in S^{2p}_R}\frac{\partial^j}{\partial\nu^j}
  \left(\e^{-|\vect{x}-\vect{s}|}\right)\mathrm{d}\vect{x}
 &=
\frac{1}{n!\,\omega_n}\int_{\hat{\vect{x}}\in S^{2p}}\frac{\partial^j}{\partial\nu^j} \left(\e^{-|R\hat{\vect{ x}}-\vect{s}|}\right) R^{2p}\mathrm{d}\hat{\vect{ x}}
  \\
  &=
\frac{1}{n!\,\omega_n}\int_{\hat{\vect{x}}\in S^{2p}}\frac{\mathrm{d}^j}{\mathrm{d}R^j} \left(\e^{-|R\hat{\vect{ x}}-\vect{s}|}\right) R^{2p}\mathrm{d}\hat{\vect{ x}}
  \\
  &=
  R^{2p}\frac{\mathrm{d}^j}{\mathrm{d}R^j}\left(\frac{1}{n!\,\omega_n}\int_{\hat{\vect{x}}\in S^{2p}} \e^{-|R\hat{\vect{ x}}-\vect{s}|} \mathrm{d}\hat{\vect{ x}}\right)
  \\
    &=
  R^{2p}\frac{\mathrm{d}^j}{\mathrm{d}R^j}\left(R^{-2p}\frac{1}{n!\,\omega_n}\int_{\hat{\vect{x}}\in S^{2p}} \e^{-|R\hat{\vect{ x}}-\vect{s}|}\,R^{2p} \mathrm{d}\hat{\vect{ x}}\right)
  \\
  &=
  R^{2p}\frac{\mathrm{d}^j}{\mathrm{d}R^j}\left(R^{-2p}\frac{1}{n!\,\omega_n}\int_{{\vect{x}}\in S^{2p}_R} \e^{-|{\vect{ x}}-\vect{s}|}\, \mathrm{d}{\vect{ x}}\right)
  \\
  &=
  R^{2p}\frac{\mathrm{d}^j}{\mathrm{d}R^j}\left(R^{-2p}\frac{(-1)^p\e^{-R}}{2^p p!}\sum_{i=0}^p \binom{p}{i}\chi_{i+p}(R) \psit_{i}(s)\right)
  \\
  &=
  \frac{(-1)^p\e^{-R}}{2^p p!}\sum_{i=0}^p \binom{p}{i} \e^{R} R^{2p}\frac{\mathrm{d}^j}{\mathrm{d}R^j}\left(R^{-2p}\e^{-R}\chi_{i+p}(R) \right)\psit_{i}(s)
  \\
  &=
  \frac{(-1)^p\e^{-R}}{2^p p!}\sum_{i=0}^p \binom{p}{i} \delta^j\chi_{i+p}(R) \psit_{i}(s),
\end{align*}
as required.

\section{Magnitude in terms of Hankel determinants}
\label{section:HankelDets}

[In this section, everything will be a function of $R$, so we will remove it from the notation, writing $\chi_i$ and $\beta_j$ instead of $\chi_i(R)$ and $\beta_j(R)$.]

The goal of this section is to prove Theorem~\ref{thm:hankel} which gives the following formula for the magnitude in terms of Hankel determinants of reverse Bessel polynomials.
\[|B_R^n|=\frac{\det [\chi_{i+j+2}]_{i,j=0}^p}{n!\,R\det[\chi_{i+j}]_{i,j=0}^p}.\]
The proof will be computational and not give much insight as to why such a compact formula in possible.

%We know from~\eqref{eqn:magbetazero} that the magnitude of an odd dimensional ball has the form $\left|\ball \right|=
%  \frac{1}{n!}\left(
%   R^n +n\beta_0 R^{n-1}
%   \right)$.
%We can rewrite this as
%\[
%  -nR^{n-1}\beta_0  +n!
%  \left|\ball \right|
%  =
%  R^n
%\]
%and include it in the linear system from Theorem~\ref{thm:linearsystem} to get the extended linear system
%\[
%\begin{pmatrix}
%\chi_p& \deltap\chi_{p}& \dots&\deltap^p\chi_p &0\\
%\chi_{p+1}& \deltap\chi_{p+1}& \dots&\deltap^p\chi_{p+1}&0\\
%\vdots&\vdots&&\vdots&\vdots\\
%\chi_{2p}& \deltap\chi_{2p} &\dots&\deltap^p\chi_{2p}&0\\
%-nR^{n-1}&0 &\dots&0&n!
%\end{pmatrix}
%\begin{pmatrix}
%\beta_0\\
%\beta_1\\
%\vdots\\
%\beta_{p}\\
%\left|\ball \right|
%\end{pmatrix}
%=
%\begin{pmatrix}
%\chi_{p+1}/R\\
%\chi_{p+2}/R\\
%\vdots\\
%\chi_{2p+1}/R\\
%R^n
%\end{pmatrix}.
%\]
We start here with the linear system~\eqref{eq:ExtendedLinearSystem} which has the magnitude as one of the unknowns.  Then
Cramer's Rule immediately gives us a formula for the magnitude in terms of determinants.
\begin{equation}
  \left|\ball \right|
  =
  \frac
  {\begin{vmatrix}
    \chi_p& \deltap\chi_{p}& \dots&\deltap^p\chi_p
      &\chi_{p+1}/R\\
    \chi_{p+1}& \deltap\chi_{p+1}& \dots&\deltap^p\chi_{p+1}
      &\chi_{p+2}/R\\
    \vdots&\vdots&&\vdots&\vdots\\
    \chi_{2p}& \deltap\chi_{2p} &\dots&\deltap^p\chi_{2p}
      &\chi_{2p+1}/R\\
    -nR^{n-1}&0 &\dots&0&R^n
  \end{vmatrix}}
  {\begin{vmatrix}
  \chi_p& \deltap\chi_{p}& \dots&\deltap^p\chi_p &0\\
\chi_{p+1}& \deltap\chi_{p+1}& \dots&\deltap^p\chi_{p+1}&0\\
\vdots&\vdots&&\vdots&\vdots\\
\chi_{2p}& \deltap\chi_{2p} &\dots&\deltap^p\chi_{2p}&0\\
-nR^{n-1}&0 &\dots&0&n!
  \end{vmatrix}}
  =: \frac{N}{D}.
\label{eq:NumeratorDenominator}
\end{equation}
We define $N$ and $D$, respectively, to be the numerator and denominator of this expression.  We can, with a bit of work, express them in terms of Hankel determinants.  First, here is a lemma, the second part of which makes good on our promise --- from after Corollary~\ref{cor:SecondCor} --- that we will see that $\deltap^j \chi_{m}$ will be written as a suitable linear combination of reverse Bessel functions.
\begin{lemma} \label{lemma:delta}
Suppose that $m\in \{0,1,2,\dots\}$, $0\le j\le m$ and $k\in \Z$ then
\begin{enumerate}
\item \label{lemma:deltafirst}
$\deltap(R^k\chi_m)= -R^{k+1}\chi_{m-1}- R^{k-1}(2p-k-1)\chi_{m}$;
\item \label{lemma:deltasecond}
$\deltap^j \chi_{m}=(-1)^j \bigl(R^j\chi_{m-j }+ %\displaystyle
\sum_{t=1}^{j}f(p,j,t) R^{j-2t}\chi_{m-j+t}\bigr)$, where $f(p,j,t)$ is an integer which is independent of $m$;
\item \label{lemma:deltathird} $\deltap^j(R^{-1} \chi_{m})=(-1)^j \bigl(R^{j-1}\chi_{m-j }+ %\displaystyle
\sum_{t=1}^{j}g(p,j,t) R^{j-2t-1}\chi_{m-j+t}\bigr)$, where $g(p,j,t)$ is an integer which is independent of $m$.
\end{enumerate}
\end{lemma}
\begin{proof}
For part~\ref{lemma:deltafirst} we will proceed by induction on $m$.  First note that
\begin{align*}
  \deltap(R^k)
  &=
  e^RR^{2p}\tfrac{\mathrm d}{\mathrm d R}\bigl( e^{-R}R^{-2p} R^k\bigr)\\
  &=
  e^RR^{2p}\bigl( -e^{-R}R^{-2p} R^k+(-2p+k)e^{-R}R^{-2p} R^{k-1}\bigr)\\
  &= -R^k - (2p-k)R^{k-1}.
\end{align*}
We have $\chi_0=1$, $\chi_1=R$ and $\chi_2=R^2+R$, so
  \begin{align*}
  \deltap(R^k\chi_1)&= \deltap(R^{k+1})=-R^{k+1}\cdot 1-(2p-k-1)R^{k-1}\cdot R\\
\intertext{and}
  \deltap(R^k\chi_2)
  &= \deltap(R^{k+2}+R^{k+1})=\deltap(R^{k+2})+\deltap(R^{k+1})\\
  &=
   -R^{k+2}-(2p-k-2)R^{k+1} -R^{k+1}-(2p-k-1)R^{k}\\
   &=
   -R^{k+1}\cdot R-(2p-k-1)R^{k-1}(R^2+R),
  \end{align*}
thus the result holds for $m=1,2$.

Now suppose the result is true for $m\le m'$, we use the recursion relation~\eqref{eq:ChiRecursion} for the reverse Bessel polynomials to show that it holds when $m=m'+1$.
\begin{align*}
  \deltap(R^k\chi_{m'+1})
  &=
  \deltap(R^k(R^2\chi_{m'-1}+(2m-1)\chi_{m'}) )\\
  &=
  \deltap(R^{k+2}\chi_{m'-1})+(2m-1)\deltap(R^k\chi_{m'}) \\
  &=
  -R^{k+3}\chi_{m'-2} - R^{k+1}(2p-k-3)\chi_{m'-1}\\
  &\quad
  -(2m-1)R^{k+1}\chi_{m'-1} - R^{k-1}(2m-1)(2p-k-1)\chi_{m'}\\
  &=
  -R^{k+1}(R^2 \chi_{m'-2}+(2m-3)\chi_{m'-1})\\
  &\quad
  -R^{k-1}(2p-k-1)(R^2\chi_{m'-1}+(2m-1)\chi_{m'})\\
  &=
  -R^{k+1}\chi_{m'}-R^{k-1}(2p-k-1)\chi_{m'+1}.
\end{align*}
Thus the result for part~\ref{lemma:deltafirst} follows by induction.

For part~\ref{lemma:deltasecond} we proceed by induction on $j$.  Here the second equality uses the inductive hypothesis and the fact that $\deltap$ is $\Z$-linear, while the third equality uses part~\ref{lemma:deltafirst}.
\begin{align*}
  \deltap^{j+1}(\chi_m) &=
  \deltap\bigl(\deltap^{j}(\chi_m)\bigr)\\
  &=
  (-1)^j\biggl[ \deltap(R^j\chi_{m-j })+ %\textstyle
    \textstyle\sum_{t=1}^{j}f(p,j,t) \deltap( R^{j-2t}\chi_{m-j+t})\biggr]\\
  &=
  (-1)^{j+1}\biggl[ R^{j+1}\chi_{m-j-1}+ R^{j-1}(2p-j-1)\chi_{m-j}+ \\
  &\qquad\qquad\quad%\displaystyle
    \textstyle\sum_{t=1}^{j}f(p,j,t)\Bigl\{ R^{j-2t+1}\chi_{m-j+t-1} +\\
   &\qquad\qquad\qquad(2p-(j-2t)-1)  R^{j-2t-1}\chi_{m-j+t}\Bigr\}\biggr]\\
  &=
  (-1)^{j+1}\biggl[ R^{j+1}\chi_{m-(j+1)}+ R^{(j+1)-2}(2p-(j+1))\chi_{m-(j+1)+1}+ \\
  &\qquad\qquad\quad%\displaystyle
    \textstyle\sum_{t=1}^{j}f(p,j,t) R^{j+1-2t}\chi_{m-(j+1)+t} +\\
  &\qquad\qquad\quad
     \textstyle\sum_{t=1}^{j}f(p,j,t)\bigl(2(p+t)-j-1\bigr)  R^{(j+1)-2t-2}\chi_{m-(j+1)+t+1}\biggr]\\
  &=
  (-1)^{j+1}\biggl[ R^{j+1}\chi_{m-(j+1)}+ \\
  &\qquad\qquad\quad
  R^{(j+1)-2}\bigl((2p-(j+1))+f(p,j,1)\bigr)\chi_{m-(j+1)+1}+ \\
  &\qquad\qquad\quad%\displaystyle
    \textstyle\sum_{t=2}^{j}f(p,j,t) R^{(j+1)-2t}\chi_{m-(j+1)+t} +\\
  &\qquad\qquad\quad
     \textstyle\sum_{s=2}^{j}f(p,j,s-1)\bigl(2(p+s)-j-3\bigr)  R^{(j+1)-2s}\chi_{m-(j+1)+s}+ \\
  &\qquad\qquad\quad
 f(p,j,j)(2p+(j+1)-2)  R^{-(j+1)}\chi_{m}\biggr]\\
  &=
  (-1)^{j+1}\biggl[ R^{j+1}\chi_{m-(j+1)}+ \\
  &\qquad\qquad\quad
  \bigl((2p-(j+1))+f(p,j,1)\bigr)R^{(j+1)-2}\chi_{m-(j+1)+1}+ \\
  &\qquad\qquad\quad%\displaystyle
    \textstyle\sum_{t=2}^{j}\Bigl(f(p,j,t)+f(p,j,t-1)(2(p+t)-j-3)\Bigr) \times\\
  &\qquad\qquad\qquad\qquad
    R^{(j+1)-2t}\chi_{m-(j+1)+t} +\\
  &\qquad\qquad\quad
  f(p,j,j)(2p+(j+1)-2)  R^{-(j+1)}\chi_{m}\biggr]\
\end{align*}
The last line means that we can define $f(p,j,t)$ inductively: for $j\ge 0$ we have $f(p, j, 0)=1$, $f(p, j, j+1)=0$ and  for $t=1,\dots, j+1$ \[f(p,j+1,t)= f(p,j,t)+(2(p+t)-j-3)f(p,j,t-1).\]

Part 3 is proved by using part 2 together with the easy-to-check Leibniz-type formula
\[
  \deltap^j(R^{-1}\chi_m)=\sum_{a=0}^j \binom{j}{a}\tfrac{\mathrm{d}^a}{\mathrm{d}R^a}(R^{-1}) \deltap^{j-a}(\chi_m).
\]
You will find that $g(p,j,t)=\sum_{\ell=0}^t\frac{j!}{(j-\ell)!}f(p, j-\ell, t-\ell)$.
\end{proof}
Armed with this lemma we can now prove that the  denominator and numerator can each be written in terms of a Hankel determinant of reverse Bessel functions.
\begin{lemma}\label{lemma:denominator}
 The denominator of the magnitude of the $n$-ball as given in~\eqref{eq:NumeratorDenominator} can be expressed in the following form:
\[D= n!\, R^{(p+1)p/2}\det[\chi_{i+j}]_{i,j=0}^p.\]
\end{lemma}
\begin{proof}
We see immediately, by expanding along the final column, that
  \[D=n!\,  {\begin{vmatrix}
  \chi_p& \deltap\chi_{p}& \dots&\deltap^p\chi_p \\
\chi_{p+1}& \deltap\chi_{p+1}& \dots&\deltap^p\chi_{p+1}\\
\vdots&\vdots&&\vdots\\
\chi_{2p}& \deltap\chi_{2p} &\dots&\deltap^p\chi_{2p}
  \end{vmatrix}}.\]
We can use part~\ref{lemma:deltafirst} of the above lemma to rewrite the second, $j=1$, column and then add $(2p-1)R^{-1}$ times the first, $j=0$, column to obtain
  \begin{align*}
  D
  &=
  n!\,
  {\begin{vmatrix}
  \chi_p& -R\chi_{p-1}- R^{-1}(2p-1)\chi_p& \deltap^2\chi_{p}&\dots&\deltap^p\chi_p \\
\chi_{p+1}& -R\chi_{p}- R^{-1}(2p-1)\chi_{p+1}&\deltap^2\chi_{p+1}& \dots&\deltap^p\chi_{p+1}\\
\vdots&\vdots&\vdots&&\vdots\\
\chi_{2p}& -R\chi_{2p-1}- R^{-1}(2p-1)\chi_{2p}& \deltap^2\chi_{2p} &\dots&\deltap^p\chi_{2p}
  \end{vmatrix}}\\[0.5em]
  &=
  n!\,
  {\begin{vmatrix}
  \chi_p& -R\chi_{p-1}& \deltap^2\chi_{p}&\dots&\deltap^p\chi_p \\
\chi_{p+1}& -R\chi_{p}&\deltap^2\chi_{p+1}& \dots&\deltap^p\chi_{p+1}\\
\vdots&\vdots&\vdots&&\vdots\\
\chi_{2p}& -R\chi_{2p-1}& \deltap^2\chi_{2p} &\dots&\deltap^p\chi_{2p}
  \end{vmatrix}}.
\end{align*}
Now we work to the right, inductively, column-by-column, using Lemma~\ref{lemma:delta} part~\ref{lemma:deltasecond} and adding a suitable linear combination of the columns to the left and find the following expression for $D$.
  \begin{align*}
  D
  &=
  n!\,
  {\begin{vmatrix}
  \chi_p& -R\chi_{p-1}& R^2\chi_{p-2}&\dots&(-R)^p\chi_0 \\
\chi_{p+1}& -R\chi_{p}&R^2\chi_{p-1}& \dots&(-R)^p\chi_{1}\\
\vdots&\vdots&\vdots&&\vdots\\
\chi_{2p}& -R\chi_{2p-1}& R^2\chi_{2p-2} &\dots&(-R)^p\chi_{p}
  \end{vmatrix}}.
\end{align*}
If we now reverse the order of the columns, then we pick up a minus sign for each pair of columns we switch and these precisely cancel out the minus signs in the matrix.  Then taking out the factor of $R^j$ in each column we find the required form of the denominator:
\[D= n!\, R^{(p+1)p/2}\,
  {\begin{vmatrix}
  \chi_0&\dots&\chi_p \\
\chi_{1}&\dots&\chi_{p+1}\\
\vdots&&\vdots\\
\chi_{p}&\dots&\chi_{2p}
  \end{vmatrix}}.
\]
\end{proof}
\begin{lemma}\label{lemma:numerator}
The numerator of the magnitude of the $n$-ball as given in~\eqref{eq:NumeratorDenominator} can be expressed in the following form:
\[N= R^{(p+1)p/2-1}\det[\chi_{i+j+2}]_{i,j=0}^p.\]
\end{lemma}
\begin{proof}
By adding $n/R$ times the final column to the first column and expanding along the bottom row we have
\[
  N=
  R^n
  \begin{vmatrix}
    \chi_p+n\chi_{p+1}/R^2& \deltap\chi_{p}& \dots&\deltap^p\chi_p
      \\
    \chi_{p+1}+n\chi_{p+2}/R^2& \deltap\chi_{p+1}& \dots&\deltap^p\chi_{p+1}
      \\
    \vdots&\vdots&&\vdots\\
    \chi_{2p}+n\chi_{2p+1}/R^2& \deltap\chi_{2p} &\dots&\deltap^p\chi_{2p}
  \end{vmatrix}.
\]
For the first, $j=0$, column we have $N_{i,0}=\chi_{p+i}+n\chi_{p+i+1}/R^2$.  By the recursion relation for the reverse Bessel polynomials we have $R^2\chi_{p+i}+(2p+2i+1)\chi_{p+i+1}=\chi_{p+i+2}$, but $n=2p+1$ and so $N_{i,0}=\chi_{p+i+2}/R^2+2i\chi_{p+i+1}/R^2$.

For the second, $j=1$, column we have
\begin{align*}
  N_{i,1}
  &=
  \deltap\chi_{p+i}
  =
  -R\chi_{p+i-1} - R^{-1}(2p-1)\chi_{p+i}\\
  &=
  -R^{-1}(\chi_{p+i+1}-(2p+2i-1)\chi_{p+i})- R^{-1}(2p-1)\chi_{p+i}\\
  &=
  -R^{-1}\chi_{p+i+1}-2iR^{-1}\chi_{p+i}.
\end{align*}

For the rest of the entries we have, by the $\Z$-linearity of $\deltap$,
\begin{align*}
  N_{i,j}
  &=
  \deltap^j\chi_{p+i}
  =
  \deltap^{j-1}\deltap\chi_{p+i}
  =
  -\deltap^{j-1}(R^{-1}\chi_{p+i+1})
  -2i\deltap^{j-1}(R^{-1}\chi_{p+i}).
\end{align*}
This gives
\[
  N=
  R^n
  \left|
  \begin{smallmatrix}
    R^{-2}\chi_{p+2}& -R^{-1}\chi_{p+1}& \dots&-\deltap^{p-1}(R^{-1}\chi_{p+1})
      \\
    R^{-2}\chi_{p+3}+2R^{-2}\chi_{p+2}& -R^{-1}\chi_{p+2}-2R^{-1}\chi_{p+1}& \dots&-\deltap^{p-1}(R^{-1}\chi_{p+2})-2\deltap^{p-1}(R^{-1}\chi_{p+1})
      \\
    R^{-2}\chi_{p+4}+4R^{-2}\chi_{p+3}& -R^{-1}\chi_{p+3}-4R^{-1}\chi_{p+2}& \dots&-\deltap^{p-1}(R^{-1}\chi_{p+3})-4\deltap^{p-1}(R^{-1}\chi_{p+2})
      \\
    \vdots&\vdots&&\vdots\\
    R^{-2}\chi_{2p+2}+2pR^{-2}\chi_{2p+1}& -R^{-1}\chi_{2p+1}-2pR^{-1}\chi_{2p}& \dots&-\deltap^{p-1}(R^{-1}\chi_{2p+1})-2p\deltap^{p-1}(R^{-1}\chi_{2p})
  \end{smallmatrix}
  \right|.
\]
Now we can work inductively from the top row down, subtracting two lots of the top row from the second row, then subtracting four lots of the resulting row from the third row, and so on downwards.  This gives
\[
  N=
  R^n
  \left|
  \begin{matrix}
    R^{-2}\chi_{p+2}& -R^{-1}\chi_{p+1}
    & -\deltap(R^{-1}\chi_{p+1})
    &\dots&-\deltap^{p-1}(R^{-1}\chi_{p+1})
      \\
    R^{-2}\chi_{p+3}& -R^{-1}\chi_{p+2}
    & -\deltap(R^{-1}\chi_{p+2})
    &
     \dots&-\deltap^{p-1}(R^{-1}\chi_{p+2})
      \\
    R^{-2}\chi_{p+4}& -R^{-1}\chi_{p+3}
    & -\deltap(R^{-1}\chi_{p+3})
    & \dots&-\deltap^{p-1}(R^{-1}\chi_{p+3})
      \\
    \vdots&\vdots&\vdots&&\vdots\\
    R^{-2}\chi_{2p+2}& -R^{-1}\chi_{2p+1}
    & -\deltap(R^{-1}\chi_{2p+1})
    & \dots&-\deltap^{p-1}(R^{-1}\chi_{2p+1})
  \end{matrix}
  \right|.
\]
Now using Lemma~\ref{lemma:delta} part~\ref{lemma:deltathird}, and the same argument as in Lemma~\ref{lemma:denominator} above, we get
\[
  N=
  R^n
  \left|
  \begin{matrix}
    R^{-2}\chi_{p+2}& -R^{-1}\chi_{p+1}
    & \chi_{p}
    &\dots&(-R)^{p-2}\chi_{2}
      \\
    R^{-2}\chi_{p+3}& -R^{-1}\chi_{p+2}
    & \chi_{p+1}
    &
     \dots&(-R)^{p-2}\chi_{3}
      \\
    R^{-2}\chi_{p+4}& -R^{-1}\chi_{p+3}
    & \chi_{p+2}
    & \dots&(-R)^{p-2}\chi_{4}
      \\
    \vdots&\vdots&\vdots&&\vdots\\
    R^{-2}\chi_{2p+2}& -R^{-1}\chi_{2p+1}
    & \chi_{2p}
    & \dots&(-R)^{p-2}\chi_{p+2}
  \end{matrix}
  \right|.
\]
Again, as in the proof above, switching the order of columns removes the minus signs, and taking out the appropriate factor of $R$ from each column gives the required form of $N$ as $R^{(p+1)p/2-1}\det[\chi_{i+j+2}]_{i,j=0}^p$.
\end{proof}
Theorem~\ref{thm:hankel}, which we were aiming to prove, on the magnitude in terms of a ratio of Hankel determinants, now follows from the expressions for the numerator and denominator in Lemmas~\ref{lemma:denominator} and~\ref{lemma:numerator}.

%\clearpage
\section{Schr\"oder path formulae for the determinants}
\label{section:SchroderPaths}
So far we know that the magnitude of an odd ball can be obtained from the ratio of two Hankel determinants of reverse Bessel polynomials.  In this section we use the fact that the generating function of reverse Bessel polynomials has a continued fraction expansion in order to get combinatorial expressions for the Hankel determinants in terms of `Schr\"oder paths'.  This will allow us to give the degrees of the numerator and denominator, to give their leading terms and to show that all of the coefficients are positive.

\subsection{Main results}
In this section we will state the combinatorial expression for the Hankel determinants of reverse Bessel polynomials (this will be proved in the next section) and then use this to prove all the remaining results about the magnitude of odd balls.

There is a beautiful theory relating Hankel determinants of sequences with a continued fraction expansion to counting weighted Schr\"oder paths.
Let us begin with some notation.  

A \define{Schr\"oder path} is a finite directed path on the lattice $\Z^2$ where  each step in the path is either an \define{ascent}, going from $(x,y)$ to $(x+1, y+1)$, a \define{descent}, going from $(x,y)$ to $(x+1,y-1)$, or a \define{flat step}, going from $(x,y)$ to $(x+2,y)$.  Four such paths are shown in Figure~\ref{figure:3Collection} (one of the paths is an empty path).

For $i\in\{0,1,2, \dots\}$ define points $P_i:=(-i, i)$ and $Q_i:=(i, i)$. % be points in the lattice $\Z^2$.  
%Define $X_k$ to be the set of
Define a \define{disjoint $k$-collection}  to be a set of disjoint Schr\"oder paths from $\{P_i\}_{i=0}^{k}$ to $\{Q_i\}_{i=0}^{k}$, where disjoint means that they have no vertices in common.   (Note that paths cannot cross as the vertices will lie on the sub-lattice of points where the sum of coordinates is even.)  A $3$-collection is shown in Figure~\ref{figure:3Collection}. %  More specifically, a Schr\"oder path from $P_i$ to $Q_i$ is a directed path in the lattice $\Z^2$ which starts at $P_i$ and ends at $Q_i$, with each step being of the form $(1,1)$, an ascent, $(1,-1)$, a descent, or $(2,0)$, a flat step.  Observe that the vertices of such a path lie on the sub-lattice where the sum of coordinates is even and that there is a single empty path from $P_0$ to $Q_0$.  A disjoint $k$-collection of paths consists of a path from $P_i$ to $Q_i$ for each $i\in \{0,1,\dots k\}$ such that the paths share no vertices.
Let $X_k$ be the set of such $k$-collections.

As an aside, it is perhaps worth mentioning, although it does not seem to be important here, that the disjoint $k$-collections are in bijection with the perfect matchings of an Aztec diamond (see~\cite{EuFu:Aztec}) and thus there are  $2^{k(k+1)/2}$  of them.
All eight disjoint $2$-collections are shown in Figure~\ref{fig:Schroeder_X_2}.

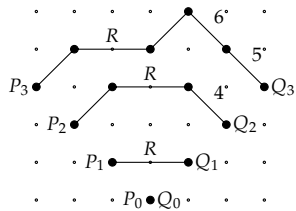
\begin{figure}
\begin{center}
\begin{tikzpicture}[scale=0.5]\scriptsize
    \foreach \i in {-3,...,3}
      \foreach \j in {0,...,5}
        \draw (\i,\j) circle[radius=0.03];
 \foreach \i in {0,...,3}{
 \node[right] at (\i, \i) {$Q_{\i}$};
 \node[left] at (-\i, \i) {$P_{\i}$};
 };
%\draw [draw=blue] (0,0) -- (4,4);
%\draw [draw=blue] (0,0) -- (-4,4);
\filldraw  (0, 0) circle [radius=0.1];
\filldraw  (-1,1) circle [radius=0.1] -- node[above ]{$R$} 
++(2,0) circle[radius=0.1];
\filldraw  (-2,2) circle [radius=0.1] -- ++(1, 1) circle [radius=0.1, fill=red]--  node[above]{$R$} ++(2,0) circle[radius=0.1] -- node[above right]{$4$}  ++(1,-1) circle[radius=0.1];
\filldraw [] (-3,3) circle [radius=0.1] -- ++(1, 1) circle [radius=0.1, fill=red]-- node[above]{$R$} ++(2,0) circle[radius=0.1] -- ++(1,1) circle[radius=0.1] -- node[above right]{$6$} ++(1,-1) circle[radius=0.1] -- node[above right]{$5$} ++(1,-1) circle[radius=0.1];
\end{tikzpicture}
\end{center}
\caption{A $3$-collection $\sigma$ with the associated weights of $W_0$ marked on (neglecting the trivial weights on the ascending steps).  Thus $W_0(\sigma, R)=120R^3$.  Subtracting two from each marked numerical weight gives $W_2(\sigma, R)=2\times 3\times 4\times R^3= 24R^3$. }
\label{figure:3Collection}
\end{figure}

A \define{path weighting} will be a way to associate a weight to each step in a path.  Let $W_0$ be the following path weighting:  associate $1$ to each ascending step, the indeterminate $R$ to each flat step and $y+1$ to each descending step which starts at height $y$.  See Figure~\ref{figure:3Collection} for an example.  For a $k$-collection $\sigma$ we define $W_0(\sigma,R)$ to be the product of the weights of all of the steps in the collection.

Similarly, let $W_2$ be the path weighting which associates $1$ to each ascending step, $R$ to each flat step and $y$ to each descending step which starts at height $y-1$.

We can now state the theorem, proved in the next subsection, which gives a combinatorial expression for each Hankel determinant of reverse Bessel polynomials that we are interested in, in terms of weighted sums of $k$-collections of Schr\"oder paths.
\begin{thm}
\label{thm:CombinatorialDeterminant}
Recalling that $\super(k):=\prod_{i=0}^k i!$ is the $k$th superfactorial, we have
\begin{align*}
  \det\left[\chi_{i+j}(R)\right]_{i,j=0}^p
  &=
  R^p\super(p)\sum_{\sigma\in X_{p-1}}W_0(\sigma, R)
  \\
  \det\left[\chi_{i+j+2}(R)\right]_{i,j=0}^p
  &=
  R^{p+1}\super(p)\sum_{\sigma\in X_{p+1}}W_2(\sigma, R).
\end{align*}
\end{thm}
Before proving this we will see some examples and prove the important consequences for magnitude.  Combining the above with Theorem~\ref{thm:hankel} we immediately get the following combinatorial expression for the magnitude of an odd ball.
\begin{cor} For $n=2p+1$ the magnitude of an $n$-ball of radius $R$ is as follows:
\[
  \left|B^n_R\right|=\frac{\sum_{\sigma\in X_{p+1}}W_2(\sigma, R)}
  {n!\,\sum_{\sigma\in X_{p-1}}W_0(\sigma, R)}
\]
\end{cor}
This leads us to define the $p$th \define{numerator polynomial} and $p$th \define{denominator polynomial}, respectively, as follows:
\[
N_p(R):=\sum_{\sigma\in X_{p+1}}W_2(\sigma, R);\qquad
D_p(R):=\sum_{\sigma\in X_{p-1}}W_0(\sigma, R).
\]

We can calculate some examples.  The $2$-collections in $X_2$ are shown in Figure~\ref{fig:Schroeder_X_2} labelled with the weightings from $W_0$.  From that we see
\begin{align*}
  D_3(R)=\sum_{\sigma\in X_{2}}W_0(\sigma, R)
  &=
  R^3+(4+4+4)R^2+(12+16+20)R+60
  =R^3+12R^2+48R+60,
\end{align*}
and, subtracting $2$ from each numerical label, we also see
\begin{align*}
  N_1(R)=\sum_{\sigma\in X_{2}}W_2(\sigma, R)
  &=
  R^3+(2+2+2)R^2+(2+4+6)R+6
  =R^3+6R^2+12R+6.
\end{align*}
You can check that these are indeed the denominator and numerator polynomials of the magnitudes $|B^7_R|$ and $|B^3_R|$, respectively, as given in the introduction.

\begin{figure}
\begin{center}
%\input{schroeder_paths}
%\\
\begin{tikzpicture}[scale=0.5]\scriptsize
\tikzstyle{every circle}=[radius=0.1]
%\draw [draw=blue,  thin] (0, 2) -- (-2.4, 2.4+2);
%\draw [draw=blue,  thin] (0, 2) -- (2.4, 2.4+2);
\foreach \i in {-2,...,2} \foreach \j in {2,...,6}
  \draw (\i,\j) circle[radius=0.03];
\filldraw(0, 2) circle;
\filldraw(-1, 3) circle -- (0, 4)  circle -- node[above right=-0.2em]{$3$} (1, 3)  circle;
\filldraw(-2, 4) circle -- (-1, 5)  circle -- (0, 6)  circle -- node[above right=-0.2em]{$5$} (1, 5)  circle -- node[above right=-0.2em]{$4$} (2, 4)  circle;
\end{tikzpicture}
\qquad
\begin{tikzpicture}[scale=0.5]\scriptsize
\tikzstyle{every circle}=[radius=0.1]
%\draw [draw=blue,  thin] (0, 2) -- (-2.4, 2.4+2);
%\draw [draw=blue,  thin] (0, 2) -- (2.4, 2.4+2);
\foreach \i in {-2,...,2} \foreach \j in {2,...,6}
  \draw (\i,\j) circle[radius=0.03];
\filldraw(0, 2) circle;
\filldraw(-1, 3) circle -- (0, 4)  circle -- node[above right=-0.2em]{$3$} (1, 3)  circle;
\filldraw(-2, 4) circle -- (-1, 5)  circle -- node[above=-0.1em]{$R$} (1, 5)  circle -- node[above right=-0.2em]{$4$} (2, 4)  circle;
\end{tikzpicture}
\qquad
\begin{tikzpicture}[scale=0.5]\scriptsize
\tikzstyle{every circle}=[radius=0.1]
%\draw [draw=blue,  thin] (0, 2) -- (-2.4, 2.4+2);
%\draw [draw=blue,  thin] (0, 2) -- (2.4, 2.4+2);
\foreach \i in {-2,...,2} \foreach \j in {2,...,6}
  \draw (\i,\j) circle[radius=0.03];
\filldraw(0, 2) circle;
\filldraw(-1, 3) circle -- node[above=-0.1em]{$R$} (1, 3)  circle;
\filldraw(-2, 4) circle -- (-1, 5)  circle -- (0, 6)  circle -- node[above right=-0.2em]{$5$} (1, 5)  circle -- node[above right=-0.2em]{$4$} (2, 4)  circle;
\end{tikzpicture}
\qquad
\begin{tikzpicture}[scale=0.5]\scriptsize
\tikzstyle{every circle}=[radius=0.1]
%\draw [draw=blue,  thin] (0, 2) -- (-2.4, 2.4+2);
%\draw [draw=blue,  thin] (0, 2) -- (2.4, 2.4+2);
\foreach \i in {-2,...,2} \foreach \j in {2,...,6}
  \draw (\i,\j) circle[radius=0.03];
\filldraw(0, 2) circle;
\filldraw(-1, 3) circle -- node[above=-0.1em]{$R$} (1, 3)  circle;
\filldraw(-2, 4) circle -- (-1, 5)  circle -- node[above=-0.1em]{$R$} (1, 5)  circle -- node[above right=-0.2em]{$4$} (2, 4)  circle;
\end{tikzpicture}
\qquad
\\[2em]
\begin{tikzpicture}[scale=0.5]\scriptsize
\tikzstyle{every circle}=[radius=0.1]
%\draw [draw=blue,  thin] (0, 2) -- (-2.4, 2.4+2);
%\draw [draw=blue,  thin] (0, 2) -- (2.4, 2.4+2);
\foreach \i in {-2,...,2} \foreach \j in {2,...,6}
  \draw (\i,\j) circle[radius=0.03];
\filldraw(0, 2) circle;
\filldraw(-1, 3) circle -- node[above=-0.1em]{$R$} (1, 3)  circle;
\filldraw(-2, 4) circle -- (-1, 5)  circle -- node[above right=-0.2em]{$4$} (0, 4)  circle -- (1, 5)  circle -- node[above right=-0.2em]{$4$} (2, 4)  circle;
\end{tikzpicture}
\qquad
\begin{tikzpicture}[scale=0.5]\scriptsize
\tikzstyle{every circle}=[radius=0.1]
%\draw [draw=blue,  thin] (0, 2) -- (-2.4, 2.4+2);
%\draw [draw=blue,  thin] (0, 2) -- (2.4, 2.4+2);
\foreach \i in {-2,...,2} \foreach \j in {2,...,6}
  \draw (\i,\j) circle[radius=0.03];
\filldraw(0, 2) circle;
\filldraw(-1, 3) circle -- node[above=-0.1em]{$R$} (1, 3)  circle;
\filldraw(-2, 4) circle -- (-1, 5)  circle -- node[above right=-0.2em]{$4$} (0, 4)  circle -- node[above=-0.1em]{$R$} (2, 4)  circle;
\end{tikzpicture}
\qquad
\begin{tikzpicture}[scale=0.5]\scriptsize
\tikzstyle{every circle}=[radius=0.1]
%\draw [draw=blue,  thin] (0, 2) -- (-2.4, 2.4+2);
%\draw [draw=blue,  thin] (0, 2) -- (2.4, 2.4+2);
\foreach \i in {-2,...,2} \foreach \j in {2,...,6}
  \draw (\i,\j) circle[radius=0.03];
\filldraw(0, 2) circle;
\filldraw(-1, 3) circle -- node[above=-0.1em]{$R$} (1, 3)  circle;
\filldraw(-2, 4) circle -- node[above=-0.1em]{$R$} (0, 4)  circle -- (1, 5)  circle -- node[above right=-0.2em]{$4$} (2, 4)  circle;
\end{tikzpicture}
\qquad
\begin{tikzpicture}[scale=0.5]\scriptsize
\tikzstyle{every circle}=[radius=0.1]
%\draw [draw=blue,  thin] (0, 2) -- (-2.4, 2.4+2);
%\draw [draw=blue,  thin] (0, 2) -- (2.4, 2.4+2);
\foreach \i in {-2,...,2} \foreach \j in {2,...,6}
  \draw (\i,\j) circle[radius=0.03];
\filldraw(0, 2) circle;
\filldraw(-1, 3) circle -- node[above=-0.1em]{$R$} (1, 3)  circle;
\filldraw(-2, 4) circle -- node[above=-0.1em]{$R$} (0, 4)  circle -- node[above=-0.1em]{$R$} (2, 4)  circle;
\end{tikzpicture}
\qquad
\end{center}
\caption{The eight $2$-collections of Schr\"oder paths in $X_2$, labelled with the weightings from $W_0$.}
\label{fig:Schroeder_X_2}
\end{figure}
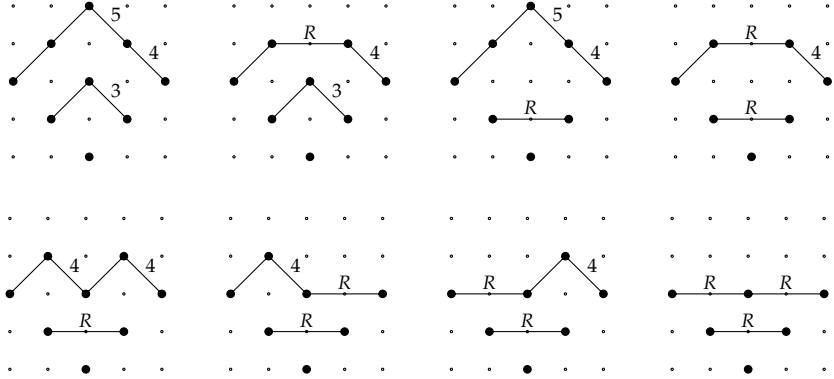

We can now prove some basic facts about the numerator and denominator which you would hope to be true if you had stared at the examples in the introduction and done some further computation.  These are not at all obvious from the Barcel\'o and Carbery algorithm nor from the determinantal formulae.
\begin{thm}
\begin{enumerate}
  \item The numerator polynomial $N_p(R)$ is  monic  of degree $\frac{(p+1)(p+2)}{2}$ with all coefficients positive.
  \item The denominator polynomial $D_p(R)$ is monic  of degree $\frac{p(p-1)}{2}$ with all coefficients positive.
  \item The constant terms are related by $N_p(0)=n!\,D_p(0)$.
\end{enumerate}
\end{thm}
\enlargethispage*{1em}
\begin{proof}
\begin{enumerate}
\item The highest degree monomials contributing to $N_p(R)$ come from $(p+1)$-collections with the maximal number of flat steps; there is precisely one of those, where all of the steps are flat, as in the last picture of Figure~\ref{fig:Schroeder_X_2}, and its weighting is $R^{(p+1)(p+2)/2}$, thus $N_p(R)$ is monic of the required degree.

Clearly each coefficient in $N_p(R)$ is non-negative.  To show that the coefficient of $R^i$ is non-zero for each $0\le i \le (p+1)(p+2)/2$ it suffices to show that there is a $(p+1)$-collection with $i$ flat steps.  We can start with the $(p+1)$-collection with $(p+1)(p+2)/2$ flat steps as above, then remove flat steps from the top line, one at a time, until we have a roof shape, i.e.~a single inverted `v'.  Then we can remove flat steps from the second-to-top line, and so on, until we have no flat steps, but just a decreasing sequences of roofs as in the first picture of Figure~\ref{fig:Schroeder_X_2}. 
\item This is proved similarly to the above.
\item The constant term of $N_p(R)$ is given by $W_2(\sigma^{p+1}_{\textrm{roof}}, R)$ where $\sigma^{p+1}_{\textrm{roof}}$ is the $(p+1)$-collection consisting of $p+1$ roofs on top of each other and no flat steps.  Similarly, the constant term of $D_p(R)$ is given by $W_0(\sigma^{p-1}_{\textrm{roof}}, R)$.  In Figure~\ref{fig:ConstantCollections} you can see that you can identify  $\sigma^{p-1}_{\textrm{roof}}$ weighted with $W_0$ as sitting inside $\sigma^{p+1}_{\textrm{roof}}$ weighted with $W_2$.  The extraneous weights are $1,2,3,\dots, 2p+1$ and thus
\[
  W_2(\sigma^{p+1}_{\textrm{roof}}, R)
  = n!\, W_0(\sigma^{p-1}_{\textrm{roof}}, R)
\]
as required.\qedhere
\end{enumerate}
\end{proof}
\begin{figure}[t]
\begin{center}
\begin{tikzpicture}[scale=0.5]\scriptsize
    \foreach \i in {-4,...,4}
      \foreach \j in {2,...,10}
        \draw (\i,\j) circle[radius=0.03];
\tikzstyle{every circle}=[radius=0.1]
%\draw [draw=black, ultra thin] (0, 2) -- (-4.4, 2.4+4);
%\draw [draw=black,  ultra thin] (0, 2) -- (4.4, 2.4+4);
\filldraw [fill=black](0, 2) circle;
\filldraw [fill=black](-1, 3) circle -- (0, 4)  circle -- node[above right=-0.2em]{$1$} (1, 3)  circle;
\filldraw [fill=black](-2, 4) circle -- (-1, 5)  circle -- (0, 6)  circle -- node[above right=-0.2em]{$3$} (1, 5)  circle -- node[above right=-0.2em]{$2$} (2, 4)  circle;
\filldraw [fill=black](-3,5) circle -- ++(1,1)  circle -- ++(1,1)  circle -- ++(1,1)  circle -- node[above right=-0.2em]{$5$} ++(1, -1) circle -- node[above right=-0.2em]{$4$} ++(1, -1) circle -- node[above right=-0.2em]{$3$} ++(1, -1)  circle;
\filldraw [fill=black](-4, 6) circle -- ++(1,1)  circle -- ++(1,1)  circle -- ++(1,1)  circle -- ++(1,1)  circle  -- node[above right=-0.2em]{$7$} ++(1, -1)  circle  -- node[above right=-0.2em]{$6$} ++(1, -1)  circle -- node[above right=-0.2em]{$5$} ++(1, -1) circle -- node[above right=-0.2em]{$4$} ++(1, -1) circle;
\draw[black, dashed] (0, 4) -- ++(2.5, 2.5) -- ++(-2.5,2.5) -- ++(-2.5,-2.5) -- ++(2.5,-2.5)  ;
\end{tikzpicture}
\end{center}
\caption{The constant terms in the polynomials $N_3(R)$ and $D_3(R)$, seen here as $W_2(\sigma^{4}_{\textrm{roof}}, R)$ and $W_0(\sigma^{2}_{\textrm{roof}}, R)$.}
\label{fig:ConstantCollections}
\end{figure}
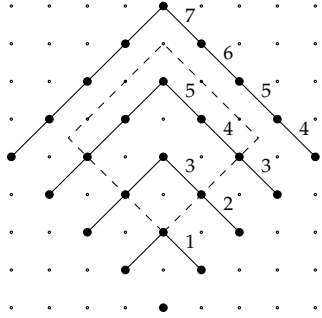

These polynomials also appear to be log-concave, as do some other polynomials arising as Hankel determinants of sequences of classical polynomials~\cite{Sokal:ContinuedFractionsBook}; however, it is not clear how to go about proving this.
%A few other observations.  There are $2^{k(k+1)/2}$ disjoint $k$-collections in $X_k$.  I do not know an obvious way to prove this, but they are in bijection with perfect matchings of Aztec diamonds.  I would expect that these polynomials are log-concave but I have no idea how to prove this.

Now we can examine the combinatorics a little deeper to obtain the leading terms of the numerator and denominator polynomials.
\begin{thm}
Writing $\tempdim := \frac{1}{2}(p+1)(p+2)$, we have
\begin{align*}
  N_p(R)&= R^{\tempdim} +\tfrac{(p+1)^2(p+2)}{2} R^{\tempdim-1} + \tfrac{p(p+1)^3(p+2)(p+3)}{8} R^{\tempdim-2}+O(R^{\tempdim-3}),\\
  D_p(R)&=
  R^{\tempdim-n} +\tfrac{(p-1)p(p+1)}{2} R^{\tempdim-n-1} + \tfrac{(p-2)(p-1)p(p+1)^3}{8} R^{\tempdim-n-2}+O(R^{\tempdim-n-3}).
\end{align*}
\end{thm}
\begin{proof}
We will prove the result for the numerator polynomial $N_p(R)$, the denominator polynomial proof is almost identical.

The subleading term in $N_p(R)$ comes from summing over $(p+1)$-collections in $X_{p+1}$ which have exactly one ascending and one descending step.  This must happen on the top path, as if both the path above and below are flat there is no room for ascending or descending.  We will split the path from $P_i$ to $Q_i$ into $i$ `zones', numbered from $1$ to $i$, each zone is of width two, so that a flat step could lie in either one or two zones.  Consider a $(p+1)$-collection with one ascending and one descending step, these occur on the top path.  Let $\zeta_1$ be the zone that the ascent lies in and $\zeta_2$ be the zone that the descent lies in.  Then \[1\le \zeta_1\le \zeta_2\le p+1,\quad \text{i.e.~}1\le \zeta_1< \zeta_2+1\le p+2,\] and such a pair of numbers determines the collection uniquely, see Figure~\ref{fig:SingleAscentCollection}, so there are $\binom{p+2}{2}$ such collections and each has $W_2$ weighting of $(p+1)R^{\kappa-1}$.  Thus the total contribution is $\binom{p+2}{2}(p+1)R^{\kappa-1}= \frac{1}{2}(p+1)^2(p+2)R^{\kappa-1}$, as required.

\begin{figure}
\begin{center}
\begin{tikzpicture}[scale=0.4]\scriptsize
\tikzstyle{every circle}=[radius=0.1]
%\draw [draw=blue] (-5,5) -- (-9, 9);
%\draw [draw=blue] (5, 5) -- (9, 9);
\foreach \x in {0, 2, 4,..., 16} {
  \draw[dotted, thick] (-8+\x, 5) -- ++(0, 4.5);
}
\node at (-1, 8) {$\zeta_1$};
\node at (3, 8) {$\zeta_2$};
\filldraw [] (-8, 8) circle node[below left]{$P_{p+1}$} -- ++(2, 0) circle -- ++(2, 0) circle -- ++(2, 0) circle -- ++(1, 1) circle -- ++(2, 0) circle -- ++(2, 0) circle -- ++(1, -1) circle -- ++(2, 0) circle -- ++(2, 0) circle node[below right]{$Q_{p+1}$}; 
\filldraw [] (-7, 7) circle node[below left]{$P_{p}$} -- ++(2, 0) circle -- ++(2, 0) circle -- ++(2, 0) circle -- ++(2, 0) circle -- ++(2, 0) circle -- ++(2, 0) circle -- ++(2, 0) circle node[below right]{$Q_{p}$}; 
\filldraw [] (-6, 6) circle node[below left]{$P_{p-1}$} -- ++(2, 0) circle -- ++(2, 0) circle -- ++(2, 0) circle -- ++(2, 0) circle -- ++(2, 0) circle -- ++(2, 0) circle node[below right]{$Q_{p-1}$}; 
\end{tikzpicture}
\end{center}
\caption{The single type of $(p+1)$-collection with one ascent and one descent, with the zones of the ascent and descent marked on.}
\label{fig:SingleAscentCollection}
\end{figure}
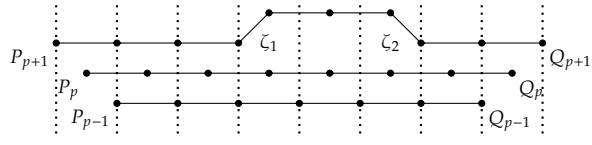

Similarly, the subsubleading term in $N_p(R)$ comes from summing over $(p+1)$-collections in $X_{p+1}$ which have exactly two ascending and two descending steps.  There are three kinds of such collection and these are pictured in Figure~\ref{fig:DoubleAscentCollection}.  In each case the zones of the ascents and descents are given by four number $\zeta_1,\zeta_2,\zeta_3,\zeta_4$ which satisfy
\begin{gather*}
  1\le \zeta_1\le \zeta_2<\zeta_3\le \zeta_4\le p+1,\\ \text{i.e.}\quad1\le \zeta_1< \zeta_2+1<\zeta_3+1< \zeta_4 +2 \le p+3.
\end{gather*}
Thus there are $\binom{p+3}{4}$ collections of each type.  The $W_2$--weightings on the three types are respectively $(p+1)(p+2)R^{\kappa-2}$, $(p+1)^2R^{\kappa-2}$ and $p(p+1)R^{\kappa-2}$.  Hence the total contribution is $\binom{p+3}{4}(p+1)\left((p+2)+(p+1)+p\right)=\frac{1}{8}p(p+1)^3(p+2)(p+3)$, as required.

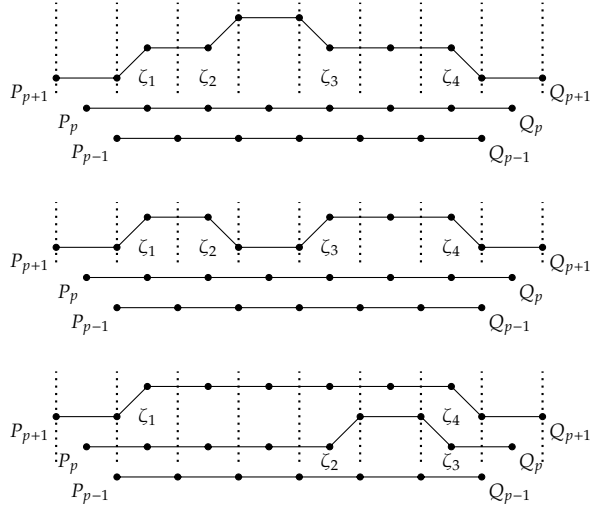
\begin{figure}
\begin{center}
\begin{tikzpicture}[scale=0.4]\scriptsize
\tikzstyle{every circle}=[radius=0.1]
%\draw [draw=blue] (-5,5) -- (-9, 9);
%\draw [draw=blue] (5, 5) -- (9, 9);
\foreach \x in {0, 2, 4,..., 16} {
  \draw[dotted, thick] (-8+\x, 7.5) -- ++(0, 3);
}
\node at (-5, 8) {$\zeta_1$};
\node at (-3, 8) {$\zeta_2$};
\node at (1, 8) {$\zeta_3$};
\node at (5, 8) {$\zeta_4$};
\filldraw [] (-8, 8) circle node[below left]{$P_{p+1}$} -- ++(2, 0) circle -- ++(1, 1) circle -- ++(2, 0) circle --  ++(1, 1) circle -- ++(2, 0) circle -- ++(1, -1) circle -- ++(2, 0) circle -- ++(2, 0) circle -- ++(1, -1) circle -- ++(2, 0) circle node[below right]{$Q_{p+1}$}; 
\filldraw [] (-7, 7) circle node[below left]{$P_{p}$} -- ++(2, 0) circle -- ++(2, 0) circle -- ++(2, 0) circle -- ++(2, 0) circle -- ++(2, 0) circle -- ++(2, 0) circle -- ++(2, 0) circle node[below right]{$Q_{p}$}; 
\filldraw [] (-6, 6) circle node[below left]{$P_{p-1}$} -- ++(2, 0) circle -- ++(2, 0) circle -- ++(2, 0) circle -- ++(2, 0) circle -- ++(2, 0) circle -- ++(2, 0) circle node[below right]{$Q_{p-1}$}; 
\end{tikzpicture}
\\[1em]
\begin{tikzpicture}[scale=0.4]\scriptsize
\tikzstyle{every circle}=[radius=0.1]
%\draw [draw=blue] (-5,5) -- (-9, 9);
%\draw [draw=blue] (5, 5) -- (9, 9);
\foreach \x in {0, 2, 4,..., 16} {
  \draw[dotted, thick] (-8+\x, 7.5) -- ++(0, 2);
}
\node at (-5, 8) {$\zeta_1$};
\node at (-3, 8) {$\zeta_2$};
\node at (1, 8) {$\zeta_3$};
\node at (5, 8) {$\zeta_4$};
\filldraw [] (-8, 8) circle node[below left]{$P_{p+1}$} -- ++(2, 0) circle -- ++(1, 1) circle -- ++(2, 0) circle --  ++(1,- 1) circle -- ++(2, 0) circle -- ++(1, 1) circle -- ++(2, 0) circle -- ++(2, 0) circle -- ++(1, -1) circle -- ++(2, 0) circle node[below right]{$Q_{p+1}$}; 
\filldraw [] (-7, 7) circle node[below left]{$P_{p}$} -- ++(2, 0) circle -- ++(2, 0) circle -- ++(2, 0) circle -- ++(2, 0) circle -- ++(2, 0) circle -- ++(2, 0) circle -- ++(2, 0) circle node[below right]{$Q_{p}$}; 
\filldraw [] (-6, 6) circle node[below left]{$P_{p-1}$} -- ++(2, 0) circle -- ++(2, 0) circle -- ++(2, 0) circle -- ++(2, 0) circle -- ++(2, 0) circle -- ++(2, 0) circle node[below right]{$Q_{p-1}$}; 
\end{tikzpicture}
\\[1em]
\begin{tikzpicture}[scale=0.4]\scriptsize
\tikzstyle{every circle}=[radius=0.1]
%\draw [draw=blue] (-5,5) -- (-9, 9);
%\draw [draw=blue] (5, 5) -- (9, 9);
\foreach \x in {0, 2, 4,..., 16} {
  \draw[dotted, thick] (-8+\x, 6.5) -- ++(0, 3);
}
\node at (-5, 8) {$\zeta_1$};
\node at (1, 6.5) {$\zeta_2$};
\node at (5, 6.5) {$\zeta_3$};
\node at (5, 8) {$\zeta_4$};
\filldraw [] (-8, 8) circle node[below left]{$P_{p+1}$} -- ++(2, 0) circle -- ++(1, 1) circle -- ++(2, 0)  circle -- ++(2, 0) circle -- ++(2, 0) circle -- ++(2, 0) circle -- ++(2, 0) circle -- ++(1, -1) circle -- ++(2, 0) circle node[below right]{$Q_{p+1}$}; 
\filldraw [] (-7, 7) circle node[below left]{$P_{p}$} -- ++(2, 0) circle -- ++(2, 0) circle -- ++(2, 0) circle -- ++(2, 0) circle -- ++(1, 1) circle -- ++(2, 0) circle --  ++(1,- 1) circle -- ++(2, 0) circle node[below right]{$Q_{p}$}; 
\filldraw [] (-6, 6) circle node[below left]{$P_{p-1}$} -- ++(2, 0) circle -- ++(2, 0) circle -- ++(2, 0) circle -- ++(2, 0) circle -- ++(2, 0) circle -- ++(2, 0) circle node[below right]{$Q_{p-1}$}; 
\end{tikzpicture}
\end{center}
\caption{The three types of $(p+1)$-collection with two ascents and two descents, with the zones of the ascents and descents marked on.}
\label{fig:DoubleAscentCollection}
\end{figure}

For the denominator polynomial $D_p(R)$ we consider essentially the same collections as above but in $X_{p-1}$ with the path weighting $W_0$.  Then there are $\binom{p-1}{2}$ collections with one ascent and one descent, each of which has $W_0$--weighting of $(p+1)R^{\kappa-n-1}$ giving the required contribution.

For the three types of collections with two ascents and two descents there are $\binom{p+1}{4}$ of them with  $W_0$--weightings respectively of $(p+1)(p+2)R^{\kappa-n-2}$, $(p+1)^2R^{\kappa-n-2}$ and $p(p+1)R^{\kappa-n-2}$, giving the required contribution.
\end{proof}

Using the above leading terms of the polynomials together with long division and the fact that $|B_R^n|=\frac{N_p(R)}{n!\,D_p(R)}$ gives the following corollary. 
\begin{cor} Asymptotically, the magnitude of an odd dimensional ball has the following form:
\[
  \left|B_R^n\right|=\frac{1}{n!}\left(R^n + \tfrac{n(n+1)}{2}R^{n-1}+ \tfrac{(n-1)n (n+1)^2}{8}R^{n-2}\right) + O(R^{n-3})
   \quad\text{as } R\to \infty.\tag*{$\square$}
\]
\end{cor}
Gimperlein and Goffeng~\cite{GimperleinGoffeng:MagnitudeFunction} have shown that the first two subleading terms in the asymptotic expansion of the magnitude of a domain in $\R^{2p+1}$ are proportional to, respectively, the volume of the boundary and the total mean curvature of the boundary.  They use the above result to pin down the constants of proportionality.

\subsection{Proof of the combinatorial formulae}
In this section we will prove Theorem~\ref{thm:CombinatorialDeterminant} giving combinatorial formulae for the Hankel determinants of the reverse Bessel polynomials.  

A classical first step in obtaining an interpretation of the Hankel determinants of a sequence is to find a continued fraction expansion of the generating function of the sequence.  Usually, Stieltjes or Jacobi-type continued fraction expansions are used, but when I gave Alan Sokal the reverse Bessel polynomials for considerations he was excited by the fact that this was the first instance of a ``combinatorially interesting'' sequence of polynomials for which there was no Stieltjes-type or Jacobi-type expansion but there was a Thron-type expansion. (See also Barry's formula given at~\cite{OEIS:A001497}.)
\begin{thm}[Sokal~\cite{Sokal:ContinuedFractionsBook}] 
\label{thm:ContinuedFraction}
There is the following Thron-type continued fraction expansion for the generating function of the reverse Bessel polynomials.
\[
 \sum_{i=0}^\infty t^i \chi_i(R) 
 =
 \dfrac{1}{1-
   \dfrac{Rt}{1-
     \dfrac{t}{1-Rt -
       \dfrac{2t}{1-Rt-
         \dfrac{3t}{1-Rt-
           \dfrac{4t}{1-\dots}}}}}}
\]
\end{thm}
\begin{proof} In \cite[Proposition~31.3]{Sokal:ContinuedFractionsBook} Sokal gives a Thron-type continued fraction for the augmented Bessel polynomials $\{Y_{i-1}(x)\}_{i=0}^\infty$, but $\chi_i(R)=R^iY_{i-1}(1/R)$ and on substitution we get the above theorem.
\end{proof}

The next step is to give a path-counting interpretation to the terms in the generating function of a continued fraction expansion.  We will need some notation first. 
%
%For sets of points $\{K_i\}_{i=1}^k,\{L_i\}_{i=1}^k\subset\Z^2$ let $X(K_1,\dots,K_k;L_1,\dots L_k)$ denote the set of disjoint collections of Schr\"oder paths $\{\gamma_i\colon K_i\to L_i\}_{i=1}^k$.  See Figure~\ref{figure:Extended3Collection} for an example.   
%
Given sequences $\alpha=(\alpha_1,\alpha_2,\dots)$ and $\delta=(\delta_1,\delta_2,\dots)$ of elements in some commutative ring $\ring$,
let  $W_{\alpha,\delta}$ be the path weighting which assigns $1$ to each ascending step, assigns $\alpha_i$ to each descending step coming down from height $i$ and assigns $\delta_{i+1}$ to each flat step at level $i$.

\begin{thm}[Fusy--Guitter--Oste--Van der Jeugt--Sokal, \cite{FusyGuitter:Quadrangulations, OsteVanderJeugt:MotzkinPaths, Sokal:ContinuedFractionsBook}]
\label{thm:ThronCFEnumeration}
If $(T_i)_{i=0}^\infty$ is a sequence of elements in a commutative ring $\ring$ whose generating function has the following Thron-type continued fraction expansion
\[
 \sum_{i=0}^\infty t^i T_i
 =
 \dfrac{1}{1- \delta_1 t -
   \dfrac{\alpha_1 t}{1-\delta_2 t -
     \dfrac{\alpha_2 t}{1-\delta_3 t -
       \dfrac{\alpha_3 t}{1-\dots}}}}~,
\]
for some
sequences $\alpha=(\alpha_1,\alpha_2,\dots)$ and $\delta=(\delta_1,\delta_2,\dots)$ of elements in the ring $\ring$,
then 
$T_i$ counts the $W_{\alpha,\delta}$--weighted Schr\"oder paths from $(a,0)$ to $(a+2i, 0)$ for any $a$:
\[T_i = \sum_{\gamma\colon (a,0) \rightsquigarrow(a+2i,0)}
 %{\{\gamma\}\in X((a,0);(a+2i,0))} 
W_{\alpha,\delta}(\gamma).\tag*{$\square$}\]
\end{thm}

Clearly in the reverse Bessel polynomials case we are working in the ring $\ring=\Z[R]$ with $\alpha=(R, 1, 2, 3, 4,\dots)$ and $\delta=(0,0,R,R,R,\dots)$.   We will write the corresponding path weighting just as $W_{\mathrm{rB}}$.

Note that it is possible to extract the combinatorial aspects of the proofs of the above two theorems and to give a continued-fraction-free proof  of the fact that the reverse Bessel polynomials have the requisite path counting interpretation.  See the blog post~\cite{Willerton:CafeBessel}.

The next ingredient is the wonderful result relating Hankel determinants to collections of non-intersecting paths in a graph.

\begin{thm}[{Karlin-McGregor-Lindstr\"om-Gessel-Viennot Lemma,  see~\cite[Chapter~29]{AignerZiegler:ProofsFromTheBook}}]
\label{thm:KMLGV}
Let $G$ be a directed acyclic graph and let $\{K_i\}_{i=0}^k$ and $\{L_j\}_{j=0}^k$ be two sets of vertices.  Let $W$ be a weighting on the edges of $G$, taking values in the commutative ring $\ring$.  Let $M_{i,j}$ for $i,j=0, \dots,k$ denote the weighted count of paths from $K_i$ to $L_j$:
\[
M_{i,j}:=\sum_{\gamma\colon K_i\rightsquigarrow L_j} W(\gamma)
\] 
Suppose that every disjoint collection of $k+1$ paths from $\{K_i\}_{i=0}^k$ and $\{L_i\}_{i=0}^k$ must connect vertex $K_i$ to $L_i$ for $i=0, \dots,k$.  Let $X(K_0,\dots,K_k;L_0,\dots,L_k)$ denote all such disjoint collections.  Then the determinant of the matrix $M$ gives the weighted count of disjoint collections of paths:
\[
\det[M_{i,j}]_{i,j=0}^k = \sum_{\sigma\in X(K_0,\dots,K_k;L_0,\dots,L_k)}W(\sigma).
\tag*{$\square$}
\]
\end{thm}

To apply this, we let $G$ be the graph whose vertices are the points in $\Z^2$ and whose edges are the three types of steps in Schr\"oder paths.
Then defining points in $\Z^2$ by $U_i:=(-2i,0)$ and $V_i:=(2i,0)$ for $i\in\{0,1,2,\dots\}$ and combining the above two theorems we obtain the following.

\begin{thm} If $(T_i)_{i=0}^\infty$ is a sequence having a Thron-type continued fraction expansion as in Theorem~\ref{thm:ThronCFEnumeration}   
then the Hankel determinants of the sequence $(T_i)_{i=0}^\infty$ count the $W_{\alpha,\delta}$--weighted disjoint collections of Schr\"oder paths:
\begin{align*}
  \det[T_{i+j}]_{i,j=0}^k &= 
  \sum_{\sigma\in X(U_0,\dots,U_k;V_0,\dots,V_k)} W_{\alpha,\delta}(\sigma)
  \\
  \intertext{and}
  \det[T_{i+j+2}]_{i,j=0}^k &= 
  \sum_{\sigma\in X(U_1,\dots,U_{k+1};V_1,\dots,V_{k+1})} W_{\alpha,\delta}(\sigma).
\tag*{$\square$}
\end{align*} 
\end{thm}

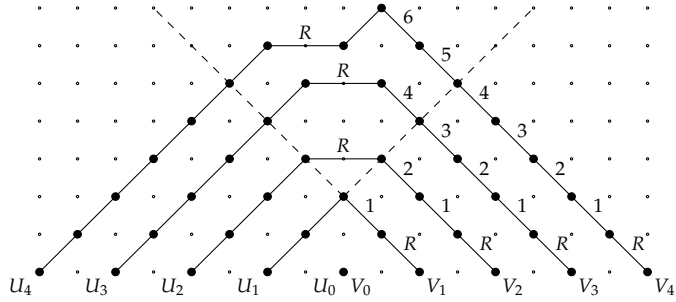
\begin{figure}
\begin{center}
\begin{tikzpicture}[scale=0.5]\scriptsize
    \foreach \i in {-8,...,8}
      \foreach \j in {-2,...,5}
        \draw (\i,\j) circle[radius=0.03];
\foreach \i in {0,...,4}{
  \node[below left] at (-2*\i, -2) {$U_{\i}$};
  \node[below right] at (2*\i, -2) {$V_{\i}$};
  };
\tikzstyle{every circle}=[radius=0.1]
\draw [draw=black, dashed, thin] (0,0) -- (5,5);
\draw [draw=black, dashed, thin] (0,0) -- (-5,5);
\filldraw [] (0,-2) circle;
\filldraw [] (-2, -2) circle -- (-1,-1) circle -- (0, 0) circle [radius=0.1] -- node[above right=-0.2em]{$1$} ++(1,-1) circle -- node[above right=-0.2em]{$R$} ++(1,-1) circle ;
\filldraw [] (-4,-2) circle -- ++(1,1) circle -- ++(1,1) circle -- (-1,1) circle [radius=0.1] -- node[above ]{$R$} 
++(2,0) circle[radius=0.1] -- node[above right=-0.2em]{$2$} ++(1,-1) circle  -- node[above right=-0.2em]{$1$} ++(1,-1) circle -- node[above right=-0.2em]{$R$} ++(1,-1) circle ;
\filldraw []  (-6,-2) circle -- ++(1,1) circle -- ++(1,1) circle -- ++(1,1) circle -- (-2,2) circle [radius=0.1] -- ++(1, 1) circle [radius=0.1, fill=red]--  node[above]{$R$} ++(2,0) circle[radius=0.1] -- node[above right=-0.2em]{$4$}  ++(1,-1) circle[radius=0.1] -- node[above right=-0.2em]{$3$} ++(1,-1) circle  -- node[above right=-0.2em]{$2$} ++(1,-1) circle  -- node[above right=-0.2em]{$1$} ++(1,-1) circle -- node[above right=-0.2em]{$R$} ++(1,-1) circle;
\filldraw [](-8,-2) circle -- ++(1,1) circle -- ++(1,1) circle -- ++(1,1) circle -- ++(1,1) circle -- (-3,3) circle [radius=0.1] -- ++(1, 1) circle [radius=0.1, fill=red]-- node[above]{$R$} ++(2,0) circle[radius=0.1] -- ++(1,1) circle[radius=0.1] -- node[above right=-0.2em]{$6$} ++(1,-1) circle[radius=0.1] -- node[above right=-0.2em]{$5$} ++(1,-1) circle[radius=0.1] -- node[above right=-0.2em]{$4$}  ++(1,-1) circle[radius=0.1] -- node[above right=-0.2em]{$3$} ++(1,-1) circle  -- node[above right=-0.2em]{$2$} ++(1,-1) circle  -- node[above right=-0.2em]{$1$} ++(1,-1) circle -- node[above right=-0.2em]{$R$} ++(1,-1) circle;
\end{tikzpicture}
\end{center}
\caption{A disjoint collection $\sigma$ in $X(U_0,\dots,U_4;V_0,\dots,V_4)$ with the path weighting $W_{\mathrm{rB}}$ marked on.  The dashed lines mark the region in which  the collection can do something interesting.}
\label{figure:Extended3Collection}
\end{figure}

We now apply this theorem to the reverse Bessel polynomials, to obtain combinatorial formulae for their Hankel determinants.  However, the formulae given in the above theorem are not quite those in Theorem~\ref{thm:CombinatorialDeterminant} as they do not have the factor taken out and the end-points are incorrect.  

Let's look first at $\det[\chi_{i+j}(R)]_{i,j=0}^p$.  By Theorem~\ref{thm:CombinatorialDeterminant}, this involves a sum over disjoint collections in $X(U_0,\dots,U_p;V_0,\dots,V_p)$ weighted by $W_{\mathrm{rB}}$.  Consider the example given in Figure~\ref{figure:Extended3Collection}.  This is typical in that everything outside the marked sector is fixed, one could have a flat step from $(-1,1)$ to $(1,1)$ but as $\delta_2=0$ this would carry a weight of $0$ and thus this collection would be ignored in the sum.  It is clear from the picture that the fixed part outside the marked sector contributes a factor of $R^p\super(p)$ to the weighting of the collection.  The part inside the marked sector can be moved down so that the apex is at the origin, then it becomes precisely a $(p-1)$-collection in $X_{p-1}$, and in this new position the path weighting is precisely $W_0$.  Thus
\[
   \det\left[\chi_{i+j}(R)\right]_{i,j=0}^p
  =
  R^p\super(p)\sum_{\sigma\in X_{p-1}}W_0(\sigma, R).
\]
as required.

The case of $\det[\chi_{i+j+2}(R)]_{i,j=0}^p$ is similar but involves one subtlety.  We have
\[
\det[\chi_{i+j+2}]_{i,j=0}^k = 
  \sum_{\sigma\in X(U_1,\dots,U_{k+1};V_1,\dots,V_{k+1})} W_{\textrm{rB}}(\sigma).
\]
We are going to take the region marked by the dashed lines in Figure~\ref{figure:NumeratorCollection}.    There are two things to consider which essentially cancel out.   Firstly, there is no vertex at $(0,0)$.  In theory we could have a horizontal path from $U_1$ to $V_1$, but the weighting assigns $0$ to such flat steps at height $0$, so we can ignore these.  However, we can have a path which dips down to $(0,0)$ as in the left hand picture of Figure~\ref{figure:NumeratorCollection}.  This carries a weight of $R$ so cannot be ignored.  Secondly, any flat step from $(-1,1)$ to $(1,1)$ carries a weight of $0$ so will be ignored in the count.  However, we have an involution on the collections in $X(U_1,\dots,U_{k+1};V_1,\dots,V_{k+1})$ which replaces a dip down to $(0,0)$ by a flat step from $(-1,1)$ to $(1,1)$ and vice versa, such that the weighting $W_\textrm{rB}$ gets replaced by the weighting in which the dip is weighted by zero and the flat step from $(-1,1)$ to $(1,1)$ is weighted by $R$.  (For example, we go from the weighted collection on the left hand side of Figure~\ref{figure:NumeratorCollection} to the one on the right hand side.)  As above we can take out as a factor all of the weights appearing outside the marked sector as these will appear in all diagrams.  On the other hand, the weighting inside the marked region is $W_2$.  This gives us
\[
  \det\left[\chi_{i+j+2}(R)\right]_{i,j=0}^p
  =
  R^{p+1}\super(p)\sum_{\sigma\in X_{p+1}}W_2(\sigma, R)
\]
as required.

\begin{figure}
\begin{center}
\begin{tikzpicture}[scale=0.4]\scriptsize
    \foreach \i in {-6,...,6}
      \foreach \j in {-2,...,3}
        \draw (\i,\j) circle[radius=0.03];
\foreach \i in {1,...,3}{
  \node[below left] at (-2*\i, -2) {$U_{\i}$};
  \node[below right] at (2*\i, -2) {$V_{\i}$};
  };
\tikzstyle{every circle}=[radius=0.1]
\draw [draw=black, dashed, thin] (0,-2) -- (5,3);
\draw [draw=black, dashed, thin] (0,-2) -- (-5,3);
%\filldraw [] (0,-2) circle;
\filldraw [] (-2, -2) circle -- (-1,-1) circle --  node[above right=-0.2em]{$R$}  ++(1,-1) circle[radius=0.1]  --  ++(1,1) circle -- node[above right=-0.2em]{$R$} ++(1,-1) circle ;
\filldraw [] (-4,-2) circle -- ++(1,1) circle -- ++(1,1) circle -- (-1,1) circle [radius=0.1] --  node[above right=-0.2em]{$2$}  ++(1,-1) circle[radius=0.1]  --  ++(1,1) circle -- node[above right=-0.2em]{$2$} ++(1,-1) circle  -- node[above right=-0.2em]{$1$} ++(1,-1) circle -- node[above right=-0.2em]{$R$} ++(1,-1) circle ;
\filldraw []  (-6,-2) circle -- ++(1,1) circle -- ++(1,1) circle -- ++(1,1) circle -- (-2,2) circle [radius=0.1] -- node[above]{$R$} ++(2,0) circle[radius=0.1] -- ++(1, 1) circle [radius=0.1, fill=red]--   node[above right=-0.2em]{$4$}  ++(1,-1) circle[radius=0.1] -- node[above right=-0.2em]{$3$} ++(1,-1) circle  -- node[above right=-0.2em]{$2$} ++(1,-1) circle  -- node[above right=-0.2em]{$1$} ++(1,-1) circle -- node[above right=-0.2em]{$R$} ++(1,-1) circle;
%\filldraw [](-8,-2) circle -- ++(1,1) circle -- ++(1,1) circle -- ++(1,1) circle -- ++(1,1) circle -- (-3,3) circle [radius=0.1] -- ++(1, 1) circle [radius=0.1, fill=red]-- node[above]{$R$} ++(2,0) circle[radius=0.1] -- ++(1,1) circle[radius=0.1] -- node[above right=-0.2em]{$6$} ++(1,-1) circle[radius=0.1] -- node[above right=-0.2em]{$5$} ++(1,-1) circle[radius=0.1] -- node[above right=-0.2em]{$4$}  ++(1,-1) circle[radius=0.1] -- node[above right=-0.2em]{$3$} ++(1,-1) circle  -- node[above right=-0.2em]{$2$} ++(1,-1) circle  -- node[above right=-0.2em]{$1$} ++(1,-1) circle -- node[above right=-0.2em]{$R$} ++(1,-1) circle;
\end{tikzpicture}
\quad
\begin{tikzpicture}[scale=0.4]\scriptsize
    \foreach \i in {-6,...,6}
      \foreach \j in {-2,...,3}
        \draw (\i,\j) circle[radius=0.03];
\foreach \i in {1,...,3}{
  \node[below left] at (-2*\i, -2) {$U_{\i}$};
  \node[below right] at (2*\i, -2) {$V_{\i}$};
  };
\tikzstyle{every circle}=[radius=0.1]
\draw [draw=black, dashed, thin] (0,-2) -- (5,3);
\draw [draw=black, dashed, thin] (0,-2) -- (-5,3);
\filldraw [] (0,-2) circle;
\filldraw [] (-2, -2) circle -- (-1,-1) circle --  node[above]{$R$}  ++(2,0) circle[radius=0.1]  --  node[above right=-0.2em]{$R$} ++(1,-1) circle ;
\filldraw [] (-4,-2) circle -- ++(1,1) circle -- ++(1,1) circle -- (-1,1) circle [radius=0.1] --  node[above right=-0.2em]{$2$}  ++(1,-1) circle[radius=0.1]  --  ++(1,1) circle -- node[above right=-0.2em]{$2$} ++(1,-1) circle  -- node[above right=-0.2em]{$1$} ++(1,-1) circle -- node[above right=-0.2em]{$R$} ++(1,-1) circle ;
\filldraw []  (-6,-2) circle -- ++(1,1) circle -- ++(1,1) circle -- ++(1,1) circle -- (-2,2) circle [radius=0.1] -- node[above]{$R$} ++(2,0) circle[radius=0.1] -- ++(1, 1) circle [radius=0.1, fill=red]--   node[above right=-0.2em]{$4$}  ++(1,-1) circle[radius=0.1] -- node[above right=-0.2em]{$3$} ++(1,-1) circle  -- node[above right=-0.2em]{$2$} ++(1,-1) circle  -- node[above right=-0.2em]{$1$} ++(1,-1) circle -- node[above right=-0.2em]{$R$} ++(1,-1) circle;
%\filldraw [](-8,-2) circle -- ++(1,1) circle -- ++(1,1) circle -- ++(1,1) circle -- ++(1,1) circle -- (-3,3) circle [radius=0.1] -- ++(1, 1) circle [radius=0.1, fill=red]-- node[above]{$R$} ++(2,0) circle[radius=0.1] -- ++(1,1) circle[radius=0.1] -- node[above right=-0.2em]{$6$} ++(1,-1) circle[radius=0.1] -- node[above right=-0.2em]{$5$} ++(1,-1) circle[radius=0.1] -- node[above right=-0.2em]{$4$}  ++(1,-1) circle[radius=0.1] -- node[above right=-0.2em]{$3$} ++(1,-1) circle  -- node[above right=-0.2em]{$2$} ++(1,-1) circle  -- node[above right=-0.2em]{$1$} ++(1,-1) circle -- node[above right=-0.2em]{$R$} ++(1,-1) circle;
\end{tikzpicture}
\end{center}
\caption{On the left, a disjoint collection $\sigma$ in $X(U_1, U_2 ,U_3;V_1, V_2 ,V_3)$, which includes a dip down to the origin, with the path weighting $W_{\mathrm{rB}}$ marked on.  On the right, the dip is replaced with a flat step with the alternative weighting (essentially $W_2$) marked on it.  The dashed lines mark the region in which collections can do something non-trivial.}
\label{figure:NumeratorCollection}
\end{figure}
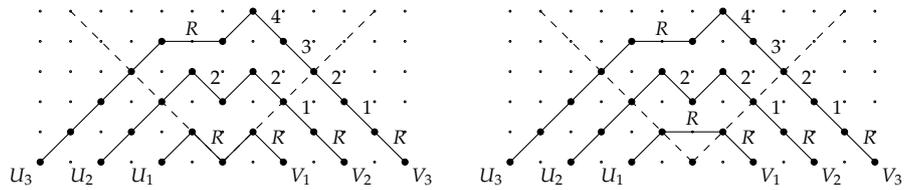

Thus Theorem~\ref{thm:CombinatorialDeterminant} is proved, and we have proved all the combinatorial facts we wanted to prove about the magnitude of odd balls.

\section*{Acknowledgments} %%  you may comment this out if no Ackno
I would like to thank the developers of SageMath, without which I would not have been able to guess many of the formulae that I have proved here; Sam Dolan for spotting a useful identity in the literature; Tom Leinster for many conversations; Mark Meckes for considerable help with understanding Bessel potential spaces; Dave Applebaum for a helpful chat;  Alan Sokal for finding the continued fraction expansion and letting me see his unfinished manuscript; Martin Lotz for mentioning log-concavity during a seminar; Heiko Gimperlein and Magnus Goffeng for some comments on a draft of this paper;  Michael Renardy for answering my question on MathOverflow~\cite{Renardy:Mathoverflow}; and the anonymous referee for helpful comments on the paper.

%%% AUTHOR:
%%% Bibliography goes here. Note that the arXiv cannot process bibtex
%%% or biber bibliographies.  Example of acceptable bibliograpy format:
\bibliographystyle{amsplain}

%% AUTHOR: You can generate such a bibliography from a .bib file by 
%% running pdflatex/bibtex/pdflatex/pdflatex and then pasting the .bbl file
%% between \begin{thebibliography} and \end{bibliography}

%%% AUTHOR: Include a short description of each author following the
%%% structure below. Use the same short tags used previously.  
%%% Use \imageat{} and \imagedot{} instead of "@" and "." in
%%% email addresses-this replaces the symbols with graphics to avoid 
%%% e-mail address harvesting from the .pdf file
\begin{dajauthors}
\begin{authorinfo}[simon]
  Simon Willerton,\\
  Senior Lecturer,\\
  School of Mathematics and Statistics,\\
  University of Sheffield,\\
  Sheffield, UK.\\
  s.willerton\imageat{}sheffield\imagedot{}ac\imagedot{}uk\\
  \url{http://www.simonwillerton.staff.shef.ac.uk/}
\end{authorinfo}
\end{dajauthors}

\end{document}